\documentclass[12pt,a4paper]{article}
\usepackage{amsmath}
\usepackage{amssymb}
\usepackage{amsthm}
\usepackage{graphicx}
\usepackage{cite}
\usepackage{hyperref}

\makeatletter
  \def\@seccntformat#1{%
    \@nameuse{@seccnt@prefix@#1}%
    \@nameuse{the#1}%
    \@nameuse{@seccnt@postfix@#1}%
    \@nameuse{@seccnt@afterskip@#1}}
  \def\@seccnt@prefix@section{}
  \def\@seccnt@postfix@section{.}
  \def\@seccnt@afterskip@section{\hspace{.5em}}
  \def\@seccnt@prefix@subsection{}
  \def\@seccnt@postfix@subsection{.}
  \def\@seccnt@afterskip@subsection{\hspace{.5em}}
\makeatother

\makeatletter
\renewcommand\section{
  \@startsection{section}{3}{\z@}%
  {-3.25ex\@plus -1ex \@minus -.2ex}%
  {1.5ex \@plus .2ex}%
  {\normalfont\normalsize\bfseries\mathversion{bold}}}
\renewcommand\subsection{
  \@startsection{subsection}{3}{\z@}%
  {-3.25ex\@plus -1ex \@minus -.2ex}%
  {1.5ex \@plus .2ex}%
  {\normalfont\normalsize\bfseries\mathversion{bold}}}
\makeatother

\makeatletter \@addtoreset{equation}{section} \makeatother
\renewcommand{\theequation}{\arabic{section}.\arabic{equation}}

\allowdisplaybreaks[1]

\let\oldthebibliography\thebibliography
\renewcommand\thebibliography[1]{
  \oldthebibliography{#1}\setlength{\itemsep}{0.4ex}}

\theoremstyle{plain}
\newtheorem{thm}{Theorem}[section]
\newtheorem{lem}[thm]{Lemma}
\newtheorem{prop}[thm]{Proposition}
\newtheorem{cor}[thm]{Corollary}

\theoremstyle{definition}
\newtheorem{algo}[thm]{Algorithm}

\newtheorem{dfn}[thm]{Definition}

\theoremstyle{remark}
\newtheorem{rem}[thm]{Remark}

\theoremstyle{plain}
\newtheorem*{thm*}{Theorem}
\newtheorem*{lem*}{Lemma}
\newtheorem*{prop*}{Proposition}
\newtheorem*{cor*}{Corollary}
\newtheorem*{conj*}{Conjecture}

\theoremstyle{definition}
\newtheorem*{algo*}{Algorithm}
\newtheorem*{ass*}{Assumption}
\newtheorem*{dfn*}{Definition}

\theoremstyle{remark}
\newtheorem*{rem*}{Remark}

\newcommand{\nn}{\nonumber}

\newcommand{\grp}[1]{\mathrm{#1}}
\newcommand{\bvec}[1]{\boldsymbol{#1}}

\newcommand{\bbR}{\mathbb{R}}
\newcommand{\bbC}{\mathbb{C}}
\newcommand{\bbZ}{\mathbb{Z}}
\newcommand{\bbH}{\mathbb{H}}
\newcommand{\bbP}{\mathbb{P}}

\newcommand{\ri}{{i}}
\newcommand{\varth}{\vartheta}

\newcommand{\tu}{\tilde{u}}
\newcommand{\tS}{\tilde{S}}
\newcommand{\tX}{\tilde{X}}
\newcommand{\cR}{R}
\newcommand{\tcR}{\tilde{R}}
\newcommand{\tpsi}{\tilde{\psi}}
\newcommand{\ha}{\hat{a}}
\newcommand{\hb}{\hat{b}}
\newcommand{\hc}{\hat{c}}
\newcommand{\hd}{\hat{d}}
\newcommand{\halpha}{\hat{\alpha}}
\newcommand{\hpsi}{\psi_J}
\newcommand{\numf}{r}

\newcommand{\vecw}{{\boldsymbol{w}}}
\newcommand{\vecz}{{\boldsymbol{z}}}
\newcommand{\veczero}{{\boldsymbol{0}}}

\newcommand{\lb}{\mathrm{b}}
\newcommand{\Jlb}{J^{E_8,\lb}_{*}}
\newcommand{\dlb}{d^\lb}
\newcommand{\bx}[1]{\makebox[1.2em][c]{#1}}

\newcommand{\mapsfrom}{\mathrel{\reflectbox{\ensuremath{\mapsto}}}}

\newcommand{\gen}[4]{{\phi^{(#1,#2)}_{#3,#4}}}
\newcommand{\ggen}[5]{{\phi^{(#1,#2)}_{#3,#4[#5]}}}
\newcommand{\tv}[3]{\bigl(#1,#2\bigr)_{#3}}

\begin{document}

%%%%%%%%%%%%%%%%%%%%%%%%%%%%%%%%%%%%%%%%%%%%%%%%%%%%%%%%%%%%%%%%%%%%%%%%
%%% Title page %%%

%%% styles for the title page %%%%%%%%%%%%%%%%%%%%%%%%%%%%%%%%%%%%%%%%%%
\def\papertitlepage{\baselineskip 3.5ex \thispagestyle{empty}}
\def\preprinumber#1#2{\hfill
\begin{minipage}{1.22in}
#1 \par\noindent #2
\end{minipage}}

%%%%%%%%%%%%%%%%%%%%%%%%%%%%%%%%%%%%%%%%%%%%%%%%%%%%%%%%%%%%%%%%%%%%%%%%
%
\papertitlepage
\setcounter{page}{0}
\preprinumber{arXiv:2410.12907}{}
\vskip 2ex
\vfill
\begin{center}
{\large\bf\mathversion{bold}
The ring of Weyl invariant $E_8$ Jacobi forms
}
\end{center}
\vfill
\baselineskip=3.5ex
\begin{center}
Kazuhiro Sakai\\

{\small
\vskip 6ex
{\it Institute for Mathematical Informatics, Meiji Gakuin University\\
1518 Kamikurata-cho, Totsuka-ku,
Yokohama 244-8539, Japan}\\
\vskip 1ex
{\tt kzhrsakai@gmail.com}

}
\end{center}
\vfill
\baselineskip=3.5ex
\begin{center} {\bf Abstract} \end{center}

We prove that the ring of Weyl invariant $E_8$ weak Jacobi forms
is isomorphic to that of joint covariants of
a binary sextic and a binary quartic form.
The ring is therefore finitely generated.
A minimal basis of generators is obtained from
that already known for the ring of covariants.

\vfill
\vspace{5ex}


%%%%%%%%%%%%%%%%%%%%%%%%%%%%%%%%%%%%%%%%%%%%%%%%%%%%%%%%%%%%%%%%%%%%%%%%
%%% Main text %%%

%%% styles for the main text %%%%%%%%%%%%%%%%%%%%%%%%%%%%%%%%%%%%%%%%%%%
\setcounter{page}{0}
\newpage
\renewcommand{\thefootnote}{\arabic{footnote}}
\setcounter{footnote}{0}
\setcounter{section}{0}
\baselineskip = 3.5ex
\pagestyle{plain}

%%%%%%%%%%%%%%%%%%%%%%%%%%%%%%%%%%%%%%%%%%%%%%%%%%%%%%%%%%%%%%%%%%%%%%%%
\section{Introduction}
%%%%%%%%%%%%%%%%%%%%%%%%%%%%%%%%%%%%%%%%%%%%%%%%%%%%%%%%%%%%%%%%%%%%%%%%

Jacobi forms are holomorphic functions that have characteristics
of both an elliptic function and a modular form \cite{EichlerZagier}.
In 1992, Wirthm\"uller studied Jacobi forms that are invariant
under the Weyl group of an irreducible root system \cite{Wirthmuller}.
He proved a Chevalley-type theorem
that the ring of Weyl invariant weak Jacobi forms
is a bigraded polynomial algebra for
any irreducible root system not of type $E_8$.
For the Weyl group of type $E_8$ (denoted by $W(E_8)$), however,
the structure of the ring has remained unclarified since then.

Meanwhile, practical studies of $W(E_8)$-invariant Jacobi forms
were developed in string theory.
Based on the seminal work \cite{Minahan:1998vr},
basic meromorphic Jacobi forms \cite{Eguchi:2002fc}
and basic holomorphic Jacobi forms \cite{Sakai:2011xg}
were constructed. By using these Jacobi forms,
the structure of the ring $J^{E_8}_{*,*}$
of $W(E_8)$-invariant weak Jacobi forms
was investigated \cite{Sakai:2017ihc,DelZotto:2017mee}.
In particular, Wang initiated the systematic study of
the space of $W(E_8)$-invariant Jacobi forms \cite{Wang:2018fil}
and proved that the ring $J^{E_8}_{*,*}$ is not a polynomial algebra.
The succeeding studies \cite{Wang:2020pzq,Sun:2021ije,Sakai:2022taq}
unveiled the complicated structure of the ring.
It has been expected that
the ring is finitely generated over $\bbC$,
but no proof was available.

In this paper we elucidate the structure of the
ring $J^{E_8}_{*,*}$ of $W(E_8)$-invariant weak Jacobi forms.
We prove that the ring is isomorphic to
the ring of joint covariants of
a binary sextic and a binary quartic form
(Theorem~\ref{thm:JandCov}).
We do this by establishing an explicit isomorphism.
The invariants and covariants of binary forms are
the central subject of classical invariant theory.
As Gordan and Hilbert proved in the 19th century,
the ring of invariants (covariants) is always finitely generated.
The above theorem therefore implies that
$J^{E_8}_{*,*}$ is finitely generated.
Moreover, a minimal basis consisting of 194 generators for
the above ring of joint covariants has been computed
by Olive \cite{Olive:2014}.
Hence, using the isomorphism we immediately obtain
a minimal basis of generators for $J^{E_8}_{*,*}$.

Needless to say, the isomorphism is not a coincidence.
The link connecting these two subjects
is a special parametrization of rational elliptic surfaces.
This parametrization,
known as the Seiberg--Witten curve for the E-string theory,
arose in string theory some time ago \cite{Eguchi:2002fc}.
It takes the form
\begin{align}
y^2=4x^3
 -\sum_{i=0}^4 a_i(\tau,\vecz)u^{4-i}x
 -\sum_{j=0}^6 b_j(\tau,\vecz)u^{6-j}.
\label{eq:SWcurveab0}
\end{align}
This is viewed as either a family of rational elliptic surfaces
or a family of elliptic curves,
depending on whether we regard
$x,y,u\in\bbC$ as coordinates
or think of $u$ as a parameter.
The coefficients $a_i,b_j$ are functions of
$\tau\in\bbH$ and $\vecz\in\bbC^8$
whose explicit forms are known \cite{Eguchi:2002fc,Sakai:2011xg}.
In fact, they are the basic
$W(E_8)$-invariant meromorphic Jacobi forms mentioned above.
On the other hand, \eqref{eq:SWcurveab0} contains
a quartic and a sextic polynomial of $u$,
which can be identified with binary forms.
This is a clue to establish the explicit isomorphism.

There are two main theorems in this paper:
Theorem~\ref{thm:intersec} and Theorem~\ref{thm:JandS}.
The former claims that the ring $J^{E_8}_{*,*}$
of $W(E_8)$-invariant weak Jacobi forms is given by
the intersection of two polynomial rings
$\cR=\bbC[a_0,a_2,a_3,a_4,b_0,\ldots,b_6]$
and $\tcR=\bbC[c_0,\ldots,c_4,d_0,d_2,\ldots,d_6]$.
Here, $c_i,d_j$ are functions related to $a_i,b_j$
in a simple manner.
We use this theorem to prove the latter:
Theorem~\ref{thm:JandS}
claims the isomorphism between $J^{E_8}_{*,*}$
and the ring of joint semiinvariants of
a binary sextic and a binary quartic.
By the Roberts isomorphism \cite{Roberts:1861},
the ring of semiinvariants and the ring of covariants
are isomorphic.
Hence, Theorem~\ref{thm:JandS} is essentially
equivalent to Theorem~\ref{thm:JandCov}
which claims the already mentioned isomorphism
between $J^{E_8}_{*,*}$ and the ring of joint covariants.

As an application of Theorem~\ref{thm:intersec},
we formulate an algorithm for constructing
all $W(E_8)$-invariant weak Jacobi forms of given weight and index.
This is more efficient than the previous algorithm
formulated in \cite{Sakai:2022taq}.
As a corollary of Theorem~\ref{thm:JandCov}, we prove
the lower bound conjecture of Sun and Wang \cite{Sun:2021ije}
that the weight of non-zero $W(E_8)$-invariant weak Jacobi forms
of index $m$ is not less than $-4m$.
By the isomorphism, this simply follows from the fact
that there are no covariants of binary forms of negative order.

Using the isomorphism we obtain
a minimal basis consisting of 194 generators
for the ring $J^{E_8}_{*,*}$ of $W(E_8)$-invariant weak Jacobi forms.
This is merely a rearrangement of Olive's results \cite{Olive:2014}
on the generators of covariants.
Nevertheless, it is still useful to present their explicit
forms, because the correspondence between
the weight and index of Jacobi forms
and the degree and order of covariants
is not transparent at all.
In fact, to establish the correspondence
we need to introduce a trigrading to $J^{E_8}_{*,*}$.
This grading is somewhat mysterious from the viewpoint of
Weyl invariant Jacobi forms.
We leave its interpretation as an open question.

The rest of this section is devoted to some
supplemental remarks on the backgrounds
of this work and related studies.

Although all results in this paper are derived
solely from the explicit forms of $a_i,b_j$
without any knowledge about their origin,
it would be informative to see
how these functions are determined
and why the $W(E_8)$ action arises.
The curve \eqref{eq:SWcurveab0} was originally 
constructed in \cite{Eguchi:2002fc} by using mirror symmetry.
There is a mirror pair of Calabi--Yau threefolds
$(X,\tX)$ of the following type
(for the details, see \cite{Minahan:1998vr}):
The threefold $X$ contains a generic rational elliptic surface $S$
as a divisor.
The threefold $\tX$ contains another rational elliptic surface $\tS$
whose complex structure is described by the equation
\eqref{eq:SWcurveab0}
(with undetermined coefficients $a_i,b_j\in\bbC$).
With the intersection bilinear form
the second homology group $H_2(S,\bbZ)$
forms the odd unimodular Lorentzian lattice
$\mathrm{I}_{9,1}=\Lambda_8\oplus\mathrm{I}_{1,1}$.
Here, $\Lambda_8$ is the $E_8$ lattice
and the odd unimodular Lorentzian lattice $\mathrm{I}_{1,1}$
is spanned by the base and the fiber class
of the elliptic fibration of $S$.
By this structure, the prepotential $F(\phi,\tau,\vecz)$,
which enumerates compact rational curves in $X$,
is written in terms of $W(E_8)$-invariant Jacobi forms.
Here, the variables $\phi$, $\tau$ and $\vecz$ represent
the K\"ahler moduli of the base, the fiber and
the curves that span $\Lambda_8$, respectively.
There is a machinery to determine $F$ as a Fourier expansion
in $e^{2\pi i\phi}$ \cite{Minahan:1998vr}.
The form of $a_i(\tau,\vecz),b_j(\tau,\vecz)$ are then determined
in such a way that they reproduce $F$ via
a partial mirror transformation.

There is a more direct way to see the action of $W(E_8)$
in \eqref{eq:SWcurveab0}: As is well known,
the Mordell--Weil lattice of a generic rational elliptic surface
(with a section) is the $E_8$ lattice \cite{Shioda:1990}.
In \cite{Eguchi:2002nx},
240 holomorphic sections of the elliptic fibration
\eqref{eq:SWcurveab0} were explicitly calculated.
These 240 sections correspond to the roots of $E_8$.
It was then clarified that
the $W(E_8)$ action on $\vecz$ indeed
induces the $W(E_8)$ action on these sections.
In other words, \eqref{eq:SWcurveab0} is a parametrization
of rational elliptic surfaces
that makes the automorphism group of the Mordell--Weil lattice
manifest.

The isomorphism described in this paper is reminiscent of
the relation between Siegel modular forms
and invariants of a binary form \cite{Igusa:1962,Igusa:1967}.
The relation has been extended to, for example,
a correspondence between vector valued Siegel modular forms
and covariants of a binary form \cite{Clery:2017}.
It will be interesting to see how extensively
such connections between automorphic forms and invariants
can be established.

%%%%%%%%%%%%%%%%%%%%%%%%%%%%%%%%%%%%%%%%%%%%%%%%%%%%%%%%%%%%%%%%%%%%%%%%
\section{Weyl invariant $E_8$ Jacobi forms}
%%%%%%%%%%%%%%%%%%%%%%%%%%%%%%%%%%%%%%%%%%%%%%%%%%%%%%%%%%%%%%%%%%%%%%%%

This section is a summary of basic results on
$W(E_8)$-invariant Jacobi forms.

Let $\bbH=\{\tau\in\bbC\,|\,\mathrm{Im}\,\tau>0\}$
be the upper half plane and set $q=e^{2\pi\ri\tau}$.
Let $E_4(\tau)$ and $E_6(\tau)$ denote the Eisenstein series
of weight $4$ and $6$, respectively.
The cusp form
\begin{align}
\Delta=\eta^{24}=\frac{1}{1728}\left(E_4^3-E_6^2\right)
\label{eq:Delta}
\end{align}
is also used frequently, 
where $\eta(\tau)$ is the Dedekind eta function.
See Appendix~\ref{app:functions} for more about the convention.
The $E_8$ lattice $\Lambda_8\subset\bbR^8$ is given by
\begin{align}
\Lambda_8
 =\left\{(x_1,\ldots,x_8)\in\bbZ^8\cup(\bbZ+\tfrac{1}{2})^8
 \mid\textstyle\sum_i x_i\equiv 0\ (\bmod\ 2)\right\}.
\end{align}
In this paper the dot product of two vectors
$\vecz=(z_1,\ldots,z_8)$ and $\vecw=(w_1,\ldots,w_8)$
is defined as $\vecz\cdot\vecw=\sum_{i}z_iw_i$.
We write $\vecz\cdot\vecz$ as $\vecz^2$.
\begin{dfn}\label{def:Jacobi}
A $W(E_8)$-invariant weak Jacobi form
of weight $k$ and index $m$ ($k\in\bbZ,\ m\in\bbZ_{\ge 0}$) 
is a holomorphic function $\varphi_{k,m}:\bbH\times\bbC^8\to \bbC$
that possesses
the following properties \cite{EichlerZagier, Wirthmuller}:
\renewcommand{\theenumi}{\roman{enumi}}
\renewcommand{\labelenumi}{(\theenumi)}
\begin{enumerate}
\item Weyl invariance:
\begin{align}
\varphi_{k,m}(\tau,w(\vecz)) = \varphi_{k,m}(\tau,\vecz),\qquad
w\in W(E_8).
\label{Weylinv}
\end{align}

\item Quasi-periodicity:
\begin{align}
\varphi_{k,m}(\tau,\vecz+\tau\bvec{\alpha}+\bvec{\beta})
=e^{-m \pi \ri (\tau\bvec{\alpha}^2+2\vecz\cdot\bvec{\alpha})}
\varphi_{k,m}(\tau,\vecz),\qquad
\bvec{\alpha},\bvec{\beta}\in\Lambda_8.
\end{align}

\item Modular transformation law:
\begin{align}\label{Modularprop}
&
\varphi_{k,m}\left(
\frac{a\tau+b}{c\tau+d}\,,\frac{\vecz}{c\tau+d}\right)
=(c\tau+d)^k\exp\left(m\pi \ri\frac{c}{c\tau+d}\,\vecz^2\right)
\varphi_{k,m}(\tau,\vecz),\\[1ex]
&
\begin{pmatrix}a&b\\ c&d\end{pmatrix} \in \grp{SL}_2(\bbZ).\nn
\end{align}

\item The function $\varphi_{k,m}(\tau,\vecz)$ admits
a Fourier expansion of the form
\begin{align}\label{Fourierform}
\varphi_{k,m}(\tau,\vecz)
=\sum_{n=0}^\infty
 \sum_{\vecw\in\Lambda_8}
 c(n,\vecw)e^{2\pi \ri\vecw\cdot\vecz}q^n.
\end{align}
\end{enumerate}
If $\varphi_{k,m}(\tau,\vecz)$ further satisfies the condition
that the coefficients $c(n,\vecw)$
of the Fourier expansion (\ref{Fourierform})
vanish unless $\vecw^2\le 2mn$,
it is called a $W(E_8)$-invariant holomorphic Jacobi form.
\end{dfn}
\begin{rem}
In this paper a Jacobi form means a weak Jacobi form
unless otherwise specified.
Let us also introduce the notion of meromorphic Jacobi forms:
If we say that $\varphi(\tau,\vecz)$ is a meromorphic Jacobi form,
we mean that $\varphi$ itself is not a Jacobi form
but there exist a modular form $f$
such that $f\varphi$ is a Jacobi form.
In this paper we only consider the case where
$f$ is a power of $E_4$ or $E_6$.
\end{rem}

Let $J^{E_8}_{k,m}$ denote
the $\bbC$-vector space of $W(E_8)$-invariant weak Jacobi forms
of weight $k$ and index $m$.
We are interested in the structure of the bigraded algebra
\begin{align}
J^{E_8}_{*,*}=\bigoplus_{k,m}J^{E_8}_{k,m}.
\end{align}

Let us next introduce
the nine basic
$W(E_8)$-invariant holomorphic Jacobi forms
constructed in \cite{Sakai:2011xg}.
They are expressed in terms of
the theta function of the $E_8$ lattice $\Lambda_8$:
\begin{align}
\begin{aligned}
\Theta(\tau,\vecz)
 :=\sum_{\vecw\in\Lambda_8}
   \exp\left(\pi \ri\tau\vecw^2
   +2\pi \ri\vecz\cdot\vecw\right)
 =\frac{1}{2}\sum_{k=1}^4\prod_{j=1}^8\varth_k(z_j,\tau).
\end{aligned}
\end{align}
Here, $\varth_k(z,\tau)$ are the Jacobi theta functions
(see Appendix~\ref{app:functions}).
The nine $W(E_8)$-invariant holomorphic Jacobi forms
are then given by \cite[Appendix A]{Sakai:2011xg}
\begin{align}
A_1(\tau,\vecz)&=\Theta(\tau,\vecz),\qquad
A_4(\tau,\vecz)=\Theta(\tau,2\vecz),\nn\\
A_m(\tau,\vecz)&=\tfrac{m^3}{m^3+1}\left(
  \Theta(m\tau,m\vecz)
  +\tfrac{1}{m^4}\mbox{$\sum_{k=0}^{m-1}$}
  \Theta(\tfrac{\tau+k}{m},\vecz)
 \right),\qquad m=2,3,5,\nn\\
B_2(\tau,\vecz)&=\tfrac{32}{5}\left(
 e_1(\tau)\Theta(2\tau,2\vecz)
 +\tfrac{1}{2^4}e_3(\tau)\Theta(\tfrac{\tau}{2},\vecz)
 +\tfrac{1}{2^4}e_2(\tau)\Theta(\tfrac{\tau+1}{2},\vecz)\right),\nn\\
B_3(\tau,\vecz)&=\tfrac{81}{80}\left(
 h_0(\tau)^2\Theta(3\tau,3\vecz)
  -\tfrac{1}{3^5}\mbox{$\sum_{k=0}^{2}$}h_0(\tfrac{\tau+k}{3})^2
  \Theta(\tfrac{\tau+k}{3},\vecz)\right),\nn\\
B_4(\tau,\vecz)&=\tfrac{16}{15}\left(
 \varth_4(2\tau)^4\Theta(4\tau,4\vecz)
 -\tfrac{1}{2^4}\varth_4(2\tau)^4
  \Theta(\tau+\tfrac{1}{2},2\vecz)\right.\nn\\
&\hspace{2em}
 \left.
 -\tfrac{1}{2^2\cdot 4^4}\mbox{$\sum_{k=0}^{3}$}
  \varth_2(\tfrac{\tau+k}{2})^4
  \Theta(\tfrac{\tau+k}{4},\vecz)\right),\nn\\
B_6(\tau,\vecz)&=\tfrac{9}{10}\left(
  h_0(\tau)^2\Theta(6\tau,6\vecz)
 +\tfrac{1}{2^4}\mbox{$\sum_{k=0}^{1}$}
  h_0(\tau+k)^2\Theta(\tfrac{3\tau+3k}{2},3\vecz)\right.\nn\\
&\hspace{2em}\left.
 -\tfrac{1}{3\cdot 3^4}\mbox{$\sum_{k=0}^{2}$}
  h_0(\tfrac{\tau+k}{3})^2\Theta(\tfrac{2\tau+2k}{3},2\vecz)\right.\nn\\
&\hspace{2em}\left.
 -\tfrac{1}{3\cdot 6^4}\mbox{$\sum_{k=0}^{5}$}
  h_0(\tfrac{\tau+k}{3})^2\Theta(\tfrac{\tau+k}{6},\vecz)\right).
\label{E8AB}
\end{align}
Here, functions $e_j(\tau)$ and $h_0(\tau)$ are defined as
\begin{align}
\begin{aligned}
e_1(\tau)&:=
 \tfrac{1}{12}\left(\varth_3(\tau)^4+\varth_4(\tau)^4\right),\\
e_2(\tau)&:=
 \tfrac{1}{12}\left(\varth_2(\tau)^4-\varth_4(\tau)^4\right),\\
e_3(\tau)&:=
 \tfrac{1}{12}\left(-\varth_2(\tau)^4-\varth_3(\tau)^4\right),\\
h_0(\tau)&:=
 \varth_3(2\tau)\varth_3(6\tau)+\varth_2(2\tau)\varth_2(6\tau).
\end{aligned}
\end{align}
Jacobi forms
$A_m,B_m$ are of weight $4,6$ respectively and of index $m$.
If we set $\vecz=\veczero$,
these Jacobi forms reduce to the Eisenstein series $E_4,E_6$:
\begin{align}
A_m(\tau,\veczero)=E_4(\tau),\qquad
B_m(\tau,\veczero)=E_6(\tau).
\end{align}
The following fact is important. 
\begin{prop}[{\cite[Theorem 4.1]{Wang:2018fil},
             \cite[Lemma 3.2]{Sun:2021ije}}]\label{prop:ABindep}
The above nine $A_i,B_j$ are algebraically independent over
the ring of modular forms $\bbC[E_4,E_6]$.
\end{prop}

A polynomial of $A_i,B_j,E_4,E_6$ is a $W(E_8)$-invariant
Jacobi form. However, the converse is not always true.
In \cite[Appendix~B.4]{DelZotto:2017mee}, Del~Zotto et al.~found that
the $W(E_8)$-invariant holomorphic Jacobi form
\begin{align}
P_{16,5}
 =864A_1^3A_2+3825A_1B_2^2-770E_6A_3B_2-840E_6A_2B_3+60E_6A_1B_4
  +21E_6^2A_5
\label{eq:P16c5}
\end{align}
vanishes at the zero points of $E_4$.
They then conjectured items (1) and (2) of the theorem below.
Sun and Wang proved not only these conjectures
but also item (3), which provides
us with a canonical expression
for every $W(E_8)$-invariant Jacobi form.
\begin{thm}[{Sun and Wang \cite[Theorem 1.1]{Sun:2021ije}}]
\label{thm:SunWang}
\renewcommand{\theenumi}{\arabic{enumi}}
\renewcommand{\labelenumi}{$(\theenumi)$}
~\\[-4ex]
\begin{enumerate}
\item The quotient $P_{16,5}/E_4$ is a $W(E_8)$-invariant
holomorphic Jacobi form of weight $12$ and index $5$.

\item For any $W(E_8)$-invariant
Jacobi form $P\in\bbC[E_6,\{A_i\},\{B_j\}]$,
if $P/E_4$ is holomorphic on $\bbH\times\bbC^8$, then
\begin{align}
\frac{P}{P_{16,5}}\in
\bbC[E_6,\{A_i\},\{B_j\}].
\end{align}

\item Every $W(E_8)$-invariant Jacobi form of index $t$
can be expressed uniquely as
\begin{align}
\frac{\sum_{j=0}^{t_1}\left({P_{16,5}}/{E_4}\right)^{t_1-j}P_j}
     {\Delta^{N_t}},
\label{eq:SunWangform}
\end{align}
where $t_1,N_t\in\bbZ_{\ge 0}$ are such that $t_1=[t/5]$,
$N_t-5t_0=0,0,1,2,3,3$ for $t-6t_0=0,1,2,3,4,5$ respectively
with $t_0=[t/6]$, $[x]$ is the integer part of $x$ and
\begin{align}
\{P_j\}_{j=0}^{t_1-1}&\in \bbC[E_6,\{A_i\},\{B_j\}],\qquad
P_{t_1}\in \bbC[E_4,E_6,\{A_i\},\{B_j\}].
\end{align}
\end{enumerate}
\end{thm}
%

%%%%%%%%%%%%%%%%%%%%%%%%%%%%%%%%%%%%%%%%%%%%%%%%%%%%%%%%%%%%%%%%%%%%%%%%
\section{Meromorphic Jacobi forms and intersection of polynomial rings}
%%%%%%%%%%%%%%%%%%%%%%%%%%%%%%%%%%%%%%%%%%%%%%%%%%%%%%%%%%%%%%%%%%%%%%%%

%%%
\subsection{Seiberg--Witten curve and its coefficients}
%%%

In this subsection we recall
the Seiberg--Witten curve for the E-string theory \cite{Eguchi:2002fc}
and the role of its coefficients $a_i,b_j$
in the study of the ring $J^{E_8}_{*,*}$ \cite{Sakai:2022taq}.
This subsection is not intended to give a mathematical formulation
of the Seiberg--Witten curve \cite{Seiberg:1994rs,Seiberg:1994aj}
nor the E-string theory \cite{Ganor:1996mu,Seiberg:1996vs};
we simply overview the concrete form of the curve
and properties of the functions $a_i,b_j$
relevant for the present study.

The Seiberg--Witten curve for the E-string theory is
a family of elliptic curves over $\bbC$.
It can be expressed in the Weierstrass form
\begin{align}
y^2=4x^3
 -\sum_{i=0}^4 a_i(\tau,\vecz)u^{4-i}x
 -\sum_{j=0}^6 b_j(\tau,\vecz)u^{6-j}.
\label{eq:SWcurveab}
\end{align}
By regarding $u$ as the affine coordinate on $\bbP^1$,
the equation \eqref{eq:SWcurveab}
describes a family of rational elliptic surfaces,
parametrized by $\tau\in\bbH$ and $\vecz\in\bbC^8$.
The coefficient functions $a_i,b_j: \bbH\times\bbC^8\to\bbC$
are not arbitrary; rather, they have to be chosen in a very special
(essentially unique) way
so that the curve reproduces the spectrum \cite{Minahan:1998vr}
of the E-string theory via a partial mirror transformation.
The existence of the curve was conjectured
in \cite{Ganor:1996xd,Ganor:1996pc}
and a fully general form was explicitly constructed
in \cite{Eguchi:2002fc}.

The coefficient functions $a_i,b_j$
play a crucial role in this paper.
Before going into details,
let us take an overview.
The variable $u$ is interpreted as the coordinate of
the moduli space of vacua and does not appear directly
in the physical spectrum.
Therefore, we can redefine $u$ by
a suitable translation $u\mapsto u+\kappa$
and set
\begin{align}
a_1=0
\end{align}
without loss of generality.
It is also customary to set
\begin{align}
a_0=\frac{E_4}{12},\qquad b_0=\frac{E_6}{216},
\end{align}
so that the elliptic fiber at $u=\infty$
is parametrized by $z\in \bbC/(2\pi\bbZ+2\pi\tau\bbZ)$
via the mapping $(x,y)=(\wp(z),\partial_z\wp(z))$.
(Here, $\wp(z)$ is the Weierstrass elliptic function.)
The other coefficients
$\{a_i\}_{i=2}^4$, $\{b_j\}_{j=1}^6$
were explicitly constructed in \cite{Eguchi:2002fc}.
By construction they possess all the properties
of $W(E_8)$-invariant weak Jacobi forms
\cite{Eguchi:2002fc,Eguchi:2002nx}
except that they have poles on $\bbH$.
Indeed, later they were expressed in terms of the above
$W(E_8)$-invariant holomorphic Jacobi forms $A_i,B_j$
and $E_4,E_6$ in \cite{Sakai:2011xg}.
The expressions are presented in Appendix~\ref{app:coeff}
for the sake of completeness.

Specifically, the coefficients $a_i,b_j$
have the following properties.
\begin{prop}[Properties of $a_i,b_j$]
\label{prop:abprop}
\renewcommand{\theenumi}{\arabic{enumi}}
\renewcommand{\labelenumi}{$(\theenumi)$}
~\\[-4ex]
\begin{enumerate}
\item Each of $a_i,b_j$ is homogeneous with respect to the bigrading:
$a_i$ is of weight $4-6i$ and index $i$;
$b_j$ is of weight $6-6j$ and index $j$.

\item The functions $a_2,a_3,a_4,b_1,\ldots,b_6$
possess all the properties of
$W(E_8)$-invariant Jacobi forms, except that
they have poles at the zero points of $E_4$.
When multiplied by a certain power of $E_4$,
they give $W(E_8)$-invariant Jacobi forms:
\begin{align}
E_4^{i-1}a_i\in J^{E_8}_{-2i,i},\qquad
E_4^j b_j\in J^{E_8}_{6-2j,j}.
\end{align}

\item The functions $a_i,b_j$ admit a Fourier expansion
of the form
\begin{align}
a_i(\tau,\vecz)=\sum_{n=0}^\infty a_i^{(n)}(\vecz)q^n,\qquad
b_j(\tau,\vecz)=\sum_{n=0}^\infty b_j^{(n)}(\vecz)q^n.
\end{align}
\end{enumerate}
\end{prop}
\begin{rem}
The leading order coefficients
$a_i^{(0)},b_j^{(0)}$
of the Fourier expansion
were explicitly computed
in \cite[Appendix~B]{Eguchi:2002fc}.
They are expressed as polynomials of characters (exponential sums)
for the Weyl orbits of the fundamental weights of $E_8$.
\end{rem}
\begin{prop}
\label{prop:abindep}
The functions
$a_0,a_2,a_3,a_4,b_0,\ldots,b_6$
are algebraically independent over $\bbC$.
\end{prop}
\begin{proof}
Using \eqref{eq:abinABfull} we obtain
\begin{align}
\left|
 \frac{\partial(a_2,a_3,a_4,b_1,b_2,b_3,b_4,b_5,b_6)}
      {\partial(A_1,A_2,B_2,A_3,B_3,A_4,B_4,A_5,B_6)}\right|
=\frac{2^{15}\cdot 5^4\cdot 7^2}{3^2\Delta^{14}E_4^2}\ne 0.
\end{align}
Combining this with
the algebraic independence of $A_i,B_j$
(Proposition~\ref{prop:ABindep}),
we see that
$a_2,a_3,a_4,b_1,\ldots,b_6$ are algebraically independent
over $\bbC[E_4,E_6]=\bbC[a_0,b_0]$.
\end{proof}
\begin{rem}\label{rem:abJacobian}
The algebraic independence of $a_i,b_j$ can be expressed more directly
via the Jacobian of Jacobi forms defined in
\cite[Proposition 2.2]{Wang:2020pzq}:
$a_2,a_3,a_4,b_1,\ldots,b_6$ are algebraically independent
over $\bbC[E_4,E_6]$ if and only if the Jacobian of them
is not identically zero.\footnote{The author
is grateful to Haowu Wang for his suggestion 
to calculate this Jacobian determinant.}
Indeed, the Jacobian is calculated as
\begin{align}
\begin{aligned}
&\left|\begin{array}{ccccccccc}
2a_2&3a_3&4a_4&b_1&2b_2&3b_3&4b_4&5b_5&6b_6\\[1ex]
\frac{1}{2\pi i}\frac{\partial a_2}{\partial z_1}&
\frac{1}{2\pi i}\frac{\partial a_3}{\partial z_1}&
\frac{1}{2\pi i}\frac{\partial a_4}{\partial z_1}&
\frac{1}{2\pi i}\frac{\partial b_1}{\partial z_1}&
\frac{1}{2\pi i}\frac{\partial b_2}{\partial z_1}&
\frac{1}{2\pi i}\frac{\partial b_3}{\partial z_1}&
\frac{1}{2\pi i}\frac{\partial b_4}{\partial z_1}&
\frac{1}{2\pi i}\frac{\partial b_5}{\partial z_1}&
\frac{1}{2\pi i}\frac{\partial b_6}{\partial z_1}\\[1ex]
\vdots&\vdots&\vdots&\vdots&\vdots&\vdots&\vdots&\vdots&\vdots\\[1ex]
\frac{1}{2\pi i}\frac{\partial a_2}{\partial z_8}&
\frac{1}{2\pi i}\frac{\partial a_3}{\partial z_8}&
\frac{1}{2\pi i}\frac{\partial a_4}{\partial z_8}&
\frac{1}{2\pi i}\frac{\partial b_1}{\partial z_8}&
\frac{1}{2\pi i}\frac{\partial b_2}{\partial z_8}&
\frac{1}{2\pi i}\frac{\partial b_3}{\partial z_8}&
\frac{1}{2\pi i}\frac{\partial b_4}{\partial z_8}&
\frac{1}{2\pi i}\frac{\partial b_5}{\partial z_8}&
\frac{1}{2\pi i}\frac{\partial b_6}{\partial z_8}
      \end{array}\right|\\[1ex]
&=-\frac{2^{16}}{3E_4}\Phi_{E_8},
\end{aligned}
\end{align}
where
\begin{align}
\Phi_{E_8}
 =\prod_{\bvec{r}\,\in\,
         \substack{\textrm{all positive}\\ \textrm{roots of $E_8$}}}
  \frac{\varth_1(\bvec{r}\cdot\vecz,\tau)}{\eta(\tau)^3}.
\end{align}
\end{rem}

\begin{prop}
\label{prop:ABinab}
The nine basic Jacobi forms
$A_i,B_j$ are polynomials of
$a_0,$ $a_2,$ $a_3,$ $a_4,$ $b_0,$ $b_1,\ldots,b_6$.
\end{prop}
\begin{proof}
By direct calculation
the expressions of $a_i,b_j$ can be inverted as in \eqref{eq:ABinab}.
\end{proof}

Let $\cR$ be the polynomial ring
\begin{align}
\cR:=\bbC[a_0,a_2,a_3,a_4,b_0,b_1,b_2,b_3,b_4,b_5,b_6].
\end{align}
Proposition~\ref{prop:abindep} means that $\cR$ is freely generated.
With emphasis on this point, we recall the main theorem of
the author's previous work:
\begin{thm}[{\cite[Theorem 3.1]{Sakai:2022taq}}]\label{thm:JE8inR}
The ring of $W(E_8)$-invariant weak Jacobi forms is
a proper subring of the polynomial ring $\cR$
\begin{align}
J^{E_8}_{*,*}\subsetneq\cR.
\end{align}
In other words,
every $W(E_8)$-invariant weak Jacobi form is expressed
uniquely as a polynomial of
$a_0,a_2,a_3,a_4,b_0,b_1,b_2,b_3,b_4,b_5,b_6$ over $\bbC$.
\end{thm}
%

%%%
\subsection{Intersection of two polynomial rings}
%%%

The goal of this subsection is to prove Theorem~\ref{thm:intersec},
which is one of the main theorems of this paper.

As mentioned in the last subsection,
we can set $a_1=0$ in the Seiberg--Witten curve
without loss of generality.
However, there is another natural choice.
Let us consider the translation of $u$
\begin{align}
\label{eq:u1trans}
u=\tu+u_1
\end{align}
such that
the Seiberg--Witten curve \eqref{eq:SWcurveab} is rewritten
in the form
\begin{align}
y^2=4x^3-\sum_{i=0}^4 c_i\tu^{4-i}x-\sum_{j=0}^6 d_j\tu^{6-j}
\label{eq:SWcurvecd}
\end{align}
with
\begin{align}
d_1=0.
\end{align}
This is done by setting
\begin{align}
u_1=-\frac{b_1}{6b_0}.
\end{align}
It is easy to see that $a_i,b_j$ and $c_k,d_l$ are related as
\begin{align}
a_i&=\sum_{j=0}^i c_j
 \begin{pmatrix} 4-j\\[1ex] 4-i\end{pmatrix}
 \left(-\frac{c_1}{4c_0}\right)^{i-j},\quad&&
b_i=\sum_{j=0}^i d_j
 \begin{pmatrix} 6-j\\[1ex] 6-i\end{pmatrix}
 \left(-\frac{c_1}{4c_0}\right)^{i-j}.
\label{eq:abincd}
\end{align}
The inverse mapping is written as
\begin{align}
c_i&=\sum_{j=0}^i a_j
 \begin{pmatrix} 4-j\\[1ex] 4-i\end{pmatrix}
 \left(-\frac{b_1}{6b_0}\right)^{i-j},&&
d_i=\sum_{j=0}^i b_j
 \begin{pmatrix} 6-j\\[1ex] 6-i\end{pmatrix}
 \left(-\frac{b_1}{6b_0}\right)^{i-j}.
\end{align}
The explicit forms of $c_i,d_j$ are presented
in \eqref{eq:cdinABfull}.
In what follows we see that
$\{c_i,d_j\}$ form another `basis',
which is complementary to $\{a_i,b_j\}$ in an interesting way.
The main point is that $c_i,d_j$ are holomorphic
at the zero points of $E_4$.
\begin{prop}[Properties of $c_i,d_j$]
\label{prop:cdprop}
\renewcommand{\theenumi}{\arabic{enumi}}
\renewcommand{\labelenumi}{$(\theenumi)$}
~\\[-4ex]
\begin{enumerate}
\item Each of $c_i,d_j$ is homogeneous with respect to
the bigrading:
$c_i$ is of weight $4-6i$ and index $i$;
$d_j$ is of weight $6-6j$ and index $j$.

\item The functions
$c_1,\ldots,c_4,d_2,\ldots,d_6$ possess all the properties of
$W(E_8)$-invariant Jacobi forms, except that
they have poles at the zero points of $E_6$.
When multiplied by a certain power of $E_6$,
they give $W(E_8)$-invariant Jacobi forms:
\begin{align}
E_6^i c_i\in J^{E_8}_{4,i},\qquad
E_6^{j-1} d_j\in J^{E_8}_{0,j}.
\end{align}

\item The functions $c_i,d_j$ admit a Fourier expansion
of the form
\begin{align}
c_i(\tau,\vecz)=\sum_{n=0}^\infty c_i^{(n)}(\vecz)q^n,\qquad
d_j(\tau,\vecz)=\sum_{n=0}^\infty d_j^{(n)}(\vecz)q^n.
\end{align}
\end{enumerate}
\end{prop}
\begin{rem}
As we see in \eqref{eq:cdinABfull},
$d_5$ contains the term $P_{16,5}/(72E_4\Delta^3)$,
which has apparent poles at the zero points of $E_4$,
but is actually holomorphic by Theorem~\ref{thm:SunWang} (1).
\end{rem}
\begin{prop}\label{prop:cdindep}
The functions
$c_0,\ldots,c_4,d_0,d_2,\ldots,d_6$
are algebraically independent over $\bbC$.
\end{prop}
\begin{proof}
By the relation \eqref{eq:abincd}
between $a_i,b_j$ and $c_k,d_l$,
the proposition directly follows from Proposition~\ref{prop:abindep}.
\end{proof}
\begin{rem}
Analogously to Remark~\ref{rem:abJacobian},
for $c_i,d_j$ we have
\begin{align}
\begin{aligned}
&\left|\begin{array}{ccccccccc}
c_1&2c_2&3c_3&4c_4&2d_2&3d_3&4d_4&5d_5&6d_6\\[1ex]
\frac{1}{2\pi i}\frac{\partial c_1}{\partial z_1}&
\frac{1}{2\pi i}\frac{\partial c_2}{\partial z_1}&
\frac{1}{2\pi i}\frac{\partial c_3}{\partial z_1}&
\frac{1}{2\pi i}\frac{\partial c_4}{\partial z_1}&
\frac{1}{2\pi i}\frac{\partial d_2}{\partial z_1}&
\frac{1}{2\pi i}\frac{\partial d_3}{\partial z_1}&
\frac{1}{2\pi i}\frac{\partial d_4}{\partial z_1}&
\frac{1}{2\pi i}\frac{\partial d_5}{\partial z_1}&
\frac{1}{2\pi i}\frac{\partial d_6}{\partial z_1}\\[1ex]
\vdots&\vdots&\vdots&\vdots&\vdots&\vdots&\vdots&\vdots&\vdots\\[1ex]
\frac{1}{2\pi i}\frac{\partial c_1}{\partial z_1}&
\frac{1}{2\pi i}\frac{\partial c_2}{\partial z_8}&
\frac{1}{2\pi i}\frac{\partial c_3}{\partial z_8}&
\frac{1}{2\pi i}\frac{\partial c_4}{\partial z_8}&
\frac{1}{2\pi i}\frac{\partial d_2}{\partial z_8}&
\frac{1}{2\pi i}\frac{\partial d_3}{\partial z_8}&
\frac{1}{2\pi i}\frac{\partial d_4}{\partial z_8}&
\frac{1}{2\pi i}\frac{\partial d_5}{\partial z_8}&
\frac{1}{2\pi i}\frac{\partial d_6}{\partial z_8}
      \end{array}\right|\\[1ex]
&=-\frac{2^{18}}{E_6}\Phi_{E_8}.
\end{aligned}
\end{align}
\end{rem}
\begin{prop}\label{prop:ABincd}
The nine basic Jacobi forms
$A_i,B_j$ are polynomials of
$c_0,\ldots,c_4,$ $d_0,$ $d_2,\ldots,d_6$.
\end{prop}
\begin{proof}
By direct calculation
the expressions of $c_i,d_j$ can be inverted as in \eqref{eq:ABincd}.
\end{proof}

Let $\tcR$ be the polynomial ring
\begin{align}
\tcR:=\bbC[c_0,c_1,c_2,c_3,c_4,d_0,d_2,d_3,d_4,d_5,d_6].
\end{align}
By Proposition~\ref{prop:cdindep}, $\tcR$ is freely generated.
Analogously to Theorem~\ref{thm:JE8inR},
the following theorem holds:
\begin{thm}\label{thm:JE8intR}
The ring of $W(E_8)$-invariant weak Jacobi forms is
a proper subring of the polynomial ring $\tcR$
\begin{align}
J^{E_8}_{*,*}\subsetneq\tcR.
\end{align}
In other words,
every $W(E_8)$-invariant weak Jacobi form is expressed
uniquely as a polynomial of
$c_0,c_1,c_2,c_3,c_4,d_0,d_2,d_3,d_4,d_5,d_6$
over $\bbC$.
\end{thm}
\begin{proof}
This theorem can be proved in the same way
as Theorem~\ref{thm:JE8inR}.
The reader is referred to \cite[Theorem 3.1]{Sakai:2022taq}.
A sketch of the proof is as follows.
By Theorem~\ref{thm:SunWang} (3) of Sun and Wang,
every $W(E_8)$-invariant Jacobi form $\varphi$
is written uniquely in the form \eqref{eq:SunWangform}.
By Proposition~\ref{prop:ABincd},
$A_i,B_j$ are polynomials of $c_k,d_l$.
The quotient $P_{16,5}/E_4$ is also written as a polynomial of
$c_k,d_l$ as
\begin{align}
\frac{P_{16,5}}{E_4}
&=-\frac{27d_0}{128}(-7110c_0c_1^5d_0^2+140\Delta c_0^2c_1^3c_2
 -21\Delta c_0c_1^5
 -1344\Delta c_0^2c_1d_2^2\nn\\
&\hspace{2em}
-9648\Delta c_0c_1^2d_0d_3+576\Delta c_0c_1c_2d_0d_2
 -6864\Delta c_1^3d_0d_2+648\Delta c_1^2c_3d_0^2\nn\\
&\hspace{2em}
 +2160\Delta c_1c_2^2d_0^2+448\Delta^2c_0c_2c_3-168\Delta^2c_1^2c_3
 -224\Delta^2c_1c_2^2+9216\Delta^2d_0d_5\nn\\
&\hspace{2em}
 -8448\Delta^2d_2d_3).
\end{align}
Note that $\Delta=c_0^3-27d_0^2$.
Thus $\Delta^{N_t}\varphi$ 
is written as a polynomial of $c_i,d_j$,
where $N_t\in\bbZ_{\ge 0}$.
Using the algebraic independence of
the constant terms (with respect to $q$)
of $c_2,c_3,c_4$ and $d_2,d_3,d_4,d_5,d_6$,
which follows from that of the constant terms of
$a_2,a_3,a_4$ and $b_2,b_3,b_4,b_5$, $b_6$
\cite[Proposition 2.5]{Sakai:2022taq},
one can inductively show that
$\Delta^n\varphi\ (n=N_t,N_t-1,\ldots,1,0)$ are polynomials
of $c_i,d_j$.
\end{proof}

We are now ready to prove
one of the main theorems of this paper.
\begin{thm}\label{thm:intersec}
The ring of $W(E_8)$-invariant weak Jacobi forms is
the intersection of two polynomial rings $\cR$ and $\tcR$
\begin{align}
J^{E_8}_{*,*}=\cR\cap\tcR.
\end{align}
In other words, the following two conditions
for a function $\varphi$ are equivalent:
\renewcommand{\theenumi}{\roman{enumi}}
\renewcommand{\labelenumi}{$\mathrm{(\theenumi})$}
~\\[-3.5ex]
\begin{enumerate}
\item $\varphi$ is a $W(E_8)$-invariant weak Jacobi form.

\vspace{-1ex}
\item $\varphi$ is
a polynomial of $a_i,b_j$ and also a polynomial of $c_i,d_j$.
\end{enumerate}
\end{thm}
\begin{proof}
It is evident
from Theorems~\ref{thm:JE8inR} and \ref{thm:JE8intR}
that $\varphi\in J^{E_8}_{*,*}\Rightarrow \varphi\in\cR\cap\tcR$.
Conversely, suppose that $\varphi\in \cR\cap\tcR$.
Then $\varphi$ possesses all the properties
of $W(E_8)$-invariant Jacobi forms,
except that it could have poles
at the zero points of $E_4$ and $E_6$.
By Proposition~\ref{prop:abprop} (2),
$\varphi$ is holomorphic at the zero points of $E_6$.
By Proposition~\ref{prop:cdprop} (2),
$\varphi$ is holomorphic at the zero points of $E_4$.
\end{proof}
%

%%%
\subsection{New algorithm}\label{sec:algorithm}
%%%

As an application of Theorem~\ref{thm:intersec}
we formulate an
algorithm for constructing all $W(E_8)$-invariant weak Jacobi forms
of given weight $k$ and index $m$. This algorithm is 
more efficient than the previous one
\cite[Algorithm 4.1]{Sakai:2022taq},
due to the simple relation \eqref{eq:abincd}
between $a_i,b_j$ and $c_k,d_l$.
\begin{algo}
\renewcommand{\theenumi}{\arabic{enumi}}
\renewcommand{\labelenumi}{(\theenumi)}
~\\[-3.5ex]
\begin{enumerate}
\item Take the most general polynomial in $a_i,b_j$
of weight $k$ and index $m$ as our ansatz.
More specifically, the ansatz is constructed
as the most general linear combination of all monomials of $a_i,b_j$
appearing in the coefficient of $x^k y^m$ in the generating series
\begin{align}
\frac{1}{\prod_{i=0,2,3,4}(1-x^{4-6i}y^ia_i)
         \prod_{j=0}^6(1-x^{6-6j}y^jb_j)}.
\end{align}

\item Substitute \eqref{eq:abincd} into the ansatz
and solve the linear equations among undetermined coefficients
so that all negative powers of $c_0$ vanish.

\item Substitute the general solution back into the original ansatz.
This gives the most general linear combination of
$W(E_8)$-invariant weak Jacobi forms of weight $k$ and index $m$.
\end{enumerate}
\end{algo}

To test the algorithm,
let us consider the graded subring of $J^{E_8}_{*,*}$ given by
\begin{align}
\Jlb
 :=\bigoplus_{m=0}^\infty J^{E_8}_{-4m,m},
\label{eq:subring}
\end{align}
which was studied in the previous work \cite{Sakai:2022taq}.
By using the new algorithm,
generators of index $m$ of $\Jlb$ for $m\le 32$
have been determined.
The results correctly reproduce the previous results for $m\le 28$.
\begin{prop}
Let $\dlb_m$ denote the number of generators of index $m$
of the graded ring $\Jlb$.
For $0\le m\le 32$, the number $\dlb_m$ is given
as in Table~$\ref{table:dlbnew}$.
\end{prop}
\begin{table}[t]
\begin{align*}
&\begin{array}
{|@{\,}c@{\,}||
@{\,}c@{\,}|@{\,}c@{\,}|@{\,}c@{\,}|@{\,}c@{\,}|@{\,}c@{\,}|
@{\,}c@{\,}|@{\,}c@{\,}|@{\,}c@{\,}|@{\,}c@{\,}|@{\,}c@{\,}|
@{\,}c@{\,}|@{\,}c@{\,}|@{\,}c@{\,}|@{\,}c@{\,}|@{\,}c@{\,}|
@{\,}c@{\,}|@{\,}c@{\,}|}\hline
m&\bx{0}&\bx{1}&\bx{2}&\bx{3}&\bx{4}&\bx{5}
&\bx{6}&\bx{7}&\bx{8}&\bx{9}&\bx{10}
&\bx{11}&\bx{12}&\bx{13}&\bx{14}&\bx{15}&\bx{16}\\ \hline
\dlb_m&0&0&0&0&1&0&2&0&1&1&2&0&3&1&3&3&3\\ \hline
\end{array}\\
&\begin{array}
{|@{\,}c@{\,}||
@{\,}c@{\,}|@{\,}c@{\,}|@{\,}c@{\,}|@{\,}c@{\,}|@{\,}c@{\,}|
@{\,}c@{\,}|@{\,}c@{\,}|@{\,}c@{\,}|@{\,}c@{\,}|@{\,}c@{\,}|
@{\,}c@{\,}|@{\,}c@{\,}|@{\,}c@{\,}|@{\,}c@{\,}|@{\,}c@{\,}|
@{\,}c@{\,}|@{\,}c@{\,}|}\hline
m
&\bx{17}&\bx{18}&\bx{19}&\bx{20}
&\bx{21}&\bx{22}&\bx{23}&\bx{24}&\bx{25}
&\bx{26}&\bx{27}&\bx{28}&\bx{29}&\bx{30}
&\bx{31}&\bx{32}&\cdots\\ \hline
\dlb_m
&3&4&3&3&4&2&3&2&3&1&1&1&2&1&1&1&\cdots\\ \hline
\end{array}
\end{align*}
\vspace{-3ex}
\caption{Number of generators
of the subring $\Jlb$ (for index $m\le 32$)}\label{table:dlbnew}
\end{table}
%

%%%%%%%%%%%%%%%%%%%%%%%%%%%%%%%%%%%%%%%%%%%%%%%%%%%%%%%%%%%%%%%%%%%%%%%%
\section{Invariants of binary forms and isomorphism}
%%%%%%%%%%%%%%%%%%%%%%%%%%%%%%%%%%%%%%%%%%%%%%%%%%%%%%%%%%%%%%%%%%%%%%%%

%%%
\subsection{Invariants of binary forms}
%%%

In this subsection we recall some definitions and useful results
from classical invariant theory
(see e.g.~\cite{Hilbert:1993,Glenn:1915}).

Let $f$ be a binary form of degree $n$.
We write it as
\begin{align}
f=f(u,v)=\sum_{i=0}^n\alpha_i u^{n-i}v^i
\end{align}
with $\alpha_i\in\bbC$, $(u,v)\in\bbC^2$.
Similarly, we consider several binary forms
$f_1,\ldots,f_\numf$ of degrees $n_1,\ldots,n_\numf$
and write them as
\begin{align}
f_k=f_k(u,v)=\sum_{i=0}^{n_k}\alpha_{k,i} u^{n_k-i}v^i.
\end{align}
We define the action of a matrix $T\in\grp{SL}_2(\bbC)$
on the variables $u,v$ by
\begin{align}
(u',v')=(t_{11}u+t_{12}v,t_{21}u+t_{22}v),\qquad
T=
\begin{pmatrix}
t_{11}&t_{12}\\
t_{21}&t_{22}
\end{pmatrix}.
\end{align}
We let the symbol $({}')$ denote the action of $T$ hereafter:
$(u',v')=T(u,v),\,\alpha'_i=T\alpha_i$, etc.
The action of $T$ on a binary form $f$
and its coefficients $\alpha_i$ is then defined by
\begin{align}
f'(u,v)=\sum_{i=0}^n\alpha'_i u^{n-i}v^i:=f(u',v').
\end{align}
\begin{dfn}[Invariants of binary forms]\label{def:invbin}
An invariant of $f$ is a homogeneous polynomial $\Phi$
in the coefficients $\alpha_i$ of $f$
that satisfies
\begin{align}
\Phi(\alpha'_i)=\Phi(\alpha_i)
\end{align}
under the action of all $T\in\grp{SL}_2(\bbC)$.
Similarly, a (joint) invariant of $f_1,\ldots,f_\numf$
is a homogeneous polynomial $\Phi$ in the coefficients $\alpha_{k,i}$
of $f_k$ that satisfies
\begin{align}
\Phi(\alpha'_{k,i})=\Phi(\alpha_{k,i})
\end{align}
under the action of all $T\in\grp{SL}_2(\bbC)$.
The degree of $\Phi$ refers to the standard degree of homogeneity
in $\alpha_i$ or $\alpha_{k,i}$.
\end{dfn}
We often omit `joint' when it does not cause any confusion.
In what follows definitions of covariants and semiinvariants 
are presented solely for several binary forms.
\begin{dfn}[Covariants of binary forms]\label{def:covbin}
A joint covariant of $f_1,\ldots,f_\numf$
of degree $d$ and order $\omega$
is a polynomial $\Psi(\alpha_{k,i};u,v)$
that satisfies the following conditions:
\renewcommand{\theenumi}{\roman{enumi}}
\renewcommand{\labelenumi}{(\theenumi)}
~\\[-3ex]
\begin{enumerate}
\item $\Psi$ is homogeneous of degree $d$
in the coefficients $\alpha_{k,i}$.

\item $\Psi$ is homogeneous of degree $\omega$
in the variables $u,v$.

\item $\Psi$ satisfies
\begin{align}
\Psi(\alpha'_{k,i};u,v)=\Psi(\alpha_{k,i};u',v')
\end{align}
under the action of all $T\in\grp{SL}_2(\bbC)$.
\end{enumerate}
\end{dfn}
\begin{dfn}[Semiinvariants of binary forms]\label{def:semibin}
A joint semiinvariant of $f_1,\ldots,f_\numf$
of degree $d$ and order $\omega$
is a polynomial $\Phi(\alpha_{k,i})$
that satisfies the following conditions:
\renewcommand{\theenumi}{\roman{enumi}}
\renewcommand{\labelenumi}{(\theenumi)}
~\\[-7ex]
\begin{enumerate}
\item $\Phi$ is homogeneous of degree $d$
in the coefficients $\alpha_{k,i}$.

\item $\Phi$ satisfies
\begin{align}
\Phi(\alpha'_{k,i})=\Phi(\alpha_{k,i})
\end{align}
under the action of all
$T=\begin{pmatrix}1&\kappa\\ 0&1\end{pmatrix}\in\grp{SL}_2(\bbC)$.

\item $\Phi$ transforms as
\begin{align}
\Phi(\alpha'_{k,i})=\lambda^\omega\Phi(\alpha_{k,i})
\end{align}
under the action of all
$T=\begin{pmatrix}\lambda&0\\ 0&\lambda^{-1}\end{pmatrix}
 \in\grp{SL}_2(\bbC)$.
\end{enumerate}
\end{dfn}

Let $V_n$ denote the $\bbC$-vector space of binary forms of degree $n$.
Let
\begin{align}
\bbC[V_{n_1}\oplus\cdots\oplus V_{n_\numf}]^{\grp{SL}_2(\bbC)}
\end{align}
denote the ring of joint invariants
of $f_1,\ldots,f_\numf$
and
\begin{align}
\bbC[V_{n_1}\oplus\cdots\oplus V_{n_\numf}\oplus
     \bbC^2]^{\grp{SL}_2(\bbC)}
\end{align}
the ring of joint covariants of $f_1,\ldots,f_\numf$.
Similarly, the ring of joint semiinvariants of $f_1,\ldots,f_\numf$
is denoted by
\begin{align}
\bbC[V_{n_1}\oplus\cdots\oplus V_{n_\numf}]^{\grp{U}_2(\bbC)},
\end{align}
where $\grp{U}_2(\bbC)$ 
is the group of upper triangular unipotent matrices
\begin{align}
\grp{U}_2(\bbC)
:=\left\{\begin{pmatrix}1&\kappa\\ 0&1\end{pmatrix}
 \,\middle|\,\kappa\in\bbC\right\}.
\end{align}

Let $\Psi$ be a joint covariant of order $\omega$.
It takes the form
\begin{align}
\Psi(\alpha_{k,i};u,v)
=\Psi_0(\alpha_{k,i}) u^\omega
 +\Psi_1(\alpha_{k,i}) u^{\omega-1}v
 +\cdots +\Psi_\omega(\alpha_{k,i}) v^\omega.
\end{align}
The leading coefficient
\begin{align}
\Psi_0(\alpha_{k,i})=\Psi(\alpha_{k,i};1,0)
\end{align}
is called the source of $\Psi$.
The following fact is well known.
\begin{thm}[{Roberts isomorphism
            \cite{Roberts:1861}}]\label{thm:Roberts}
The source of a covariant is a semiinvariant.
Moreover, the mapping
\begin{align}
\begin{aligned}
\bbC[V_{n_1}\oplus\cdots\oplus V_{n_\numf}
     \oplus\bbC^2]^{\grp{SL}_2(\bbC)}
&\ \to\ 
\bbC[V_{n_1}\oplus\cdots\oplus V_{n_\numf}]^{\grp{U}_2(\bbC)};\\
\Psi(\alpha_{k,i};u,v)&\ \mapsto\ 
\Psi_0(\alpha_{k,i})
\end{aligned}
\end{align}
is an isomorphism.
\end{thm}
The inverse mapping is given as follows.
For any semiinvariant $\Psi_0$ of order $\omega$,
the corresponding covariant $\Psi$ is recovered by
\begin{align}\label{eq:Robinv}
\Psi(\alpha_{k,i};u,v)
=u^\omega\Psi_0(\halpha_{k,i}(u,v))
\end{align}
with
\begin{align}
\halpha_{k,i}(u,v)
=\sum_{j=i}^{n_k}\alpha_{k,j}
 \begin{pmatrix}j\\[1ex] i\end{pmatrix}
 \left(\frac{v}{u}\right)^{j-i}.
\end{align}
It is clear that
the degree and order of covariants agree with
those of semiinvariants, respectively.

A covariant of order $0$ is an invariant, and vice versa.
This leads to the following theorem.
It is a trivial corollary of Theorem~\ref{thm:Roberts},
but is worth noting:
\begin{thm}
A semiinvariant of order $0$ is an invariant, and vice versa.
\end{thm}

Before closing this subsection, let us take
some examples of semiinvariants.
The fact that these are semiinvariants
is well known in classical invariant theory
(see e.g.~\cite[Section~8.1.7]{Glenn:1915}).
Nevertheless, for our later purposes it is useful to present
their explicit forms and give a concrete proof.

Let us define $\gamma_{k,i}^{(m)}$ by the relations
\begin{align}
\sum_{i=0}^{n_k}\gamma_{k,i}^{(m)}u^{n_k-i}
=f_k\left(u-\frac{\alpha_{m,1}}{n_m\alpha_{m,0}},1\right).
\label{eq:alkidef}
\end{align}
Explicitly, $\gamma_{k,i}^{(m)}$ are given by
\begin{align}
\gamma_{k,i}^{(m)}
=\sum_{j=0}^i\alpha_{k,j}
 \begin{pmatrix} n_i-j\\[1ex] n_i-i\end{pmatrix}
 \left(-\frac{\alpha_{m,1}}{n_m\alpha_{m,0}}\right)^{i-j}.
\end{align}
Note that $\gamma_{k,1}^{(k)}=0$ by construction.
\begin{prop}\label{prop:alsemi}
The polynomials
$\alpha_{k,0}^{i-1}\gamma_{k,i}^{(k)}$
are semiinvariants of $f_k$.
The polynomials
$\alpha_{m,0}^i\gamma_{k,i}^{(m)}\ (k\ne m)$
are joint semiinvariants of $f_k$ and $f_m$.
\end{prop}
\begin{proof}
Clearly, they are homogeneous polynomials of $\alpha_{k,i}$.
Since a semiinvariant of $f_k$ is
a joint semiinvariant of $f_k$ and $f_m$,
it is enough to show that $\alpha_{m,0}^i\gamma_{k,i}^{(m)}$
satisfy the conditions (ii) and (iii) of Definition~\ref{def:semibin}.
Under the transformation $(u',v')=(u+\kappa v,v)$,
the coefficients $\alpha_{k,i}$ transform as
\begin{align}
\sum_{i=0}^{n_k}\alpha'_{k,i}u^{n_k-i}
=\sum_{i=0}^{n_k}\alpha_{k,i}(u+\kappa)^{n_k-i}.
\label{eq:alkitrans}
\end{align}
This is equivalent to
\begin{align}
\begin{aligned}
\sum_{i=0}^{n_k}\alpha'_{k,i}
 \left(u-\frac{\alpha'_{m,1}}{n_m\alpha'_{m,0}}\right)^{n_k-i}
&=\sum_{i=0}^{n_k}\alpha_{k,i}
 \left(u-\frac{\alpha'_{m,1}}{n_m\alpha'_{m,0}}+\kappa\right)^{n_k-i}\\
&=\sum_{i=0}^{n_k}\alpha_{k,i}
 \left(u-\frac{\alpha_{m,1}}{n_m\alpha_{m,0}}\right)^{n_k-i}.
\end{aligned}
\end{align}
In the last equality we have used
$\alpha'_{m,0}=\alpha_{m,0},\ 
\alpha'_{m,1}=\alpha_{m,1}+n_m\kappa\alpha_{m,0}$,
which also follow from \eqref{eq:alkitrans}.
Comparing this with \eqref{eq:alkidef},
we obtain
$\gamma_{k,i}^{(m)}(\alpha'_{l,j})=\gamma_{k,i}^{(m)}(\alpha_{l,j})$.
Next, under the transformation $(u',v')=(\lambda u,\lambda^{-1}v)$
we see that
\begin{align}
\alpha'_{k,i}=\lambda^{n_k-2i}\alpha_{k,i},\qquad
{\gamma'}_{k,i}^{(m)}=\lambda^{n_k-2i}\gamma_{k,i}^{(m)}.
\label{eq:alorder}
\end{align}
From this we obtain
$(\alpha'_{m,0})^i{\gamma'}_{k,i}^{(m)}
=\lambda^{n_k+(n_m-2)i}\alpha_{m,0}^i\gamma_{k,i}^{(m)}$.
\end{proof}
%

%%%
\subsection{Trigrading}
%%%

In this subsection we introduce a trigrading to
the ring $J^{E_8}_{*,*}$ of $W(E_8)$-invariant weak Jacobi forms.
This is necessary for the explicit isomorphism
that we construct in the next subsection.

Let $f$ be a binary quartic and $g$ a binary sextic
\begin{align}
f=\sum_{i=0}^4\alpha_iu^{4-i}v^i,\qquad g=\sum_{i=0}^6\beta_iu^{6-i}v^i.
\end{align}
We consider joint covariants $\Psi(\alpha_i,\beta_j;u,v)$
and joint semiinvariants $\Phi(\alpha_i,\beta_j)$
of $f$ and $g$.
The ring of joint covariants and the ring of joint semiinvariants
are denoted respectively by
\begin{align}
\bbC[V_4\oplus V_6\oplus\bbC^2]^{\grp{SL_2(\bbC)}},\qquad
\bbC[V_4\oplus V_6]^{\grp{U}_2(\bbC)}.
\end{align}
Since we always deal with $f$ and $g$ of fixed degrees,
for simplicity's sake we often omit the variable $v$
using the trivial identification of binary forms with polynomials
\begin{align}
f(u,1)=f(u),\qquad g(u,1)=g(u).
\end{align}

Let us first introduce the notion of refined degrees.
\begin{dfn}\label{def:degalbe}
We say that $\Xi\in
 \bbC[\alpha_0,\ldots,\alpha_4,\beta_0,\ldots,\beta_6,
 \alpha_0^{-1},\beta_0^{-1},u,v]$
is of degrees $(d_\alpha,d_\beta)$
if $\Xi$ is 
homogeneous of degree $d_\alpha$ in $\alpha_i$
and homogeneous of degree $d_\beta$ in $\beta_i$.
\end{dfn}
\begin{rem}
Clearly, a joint covariant (semiinvariant) of degrees
$(d_\alpha,d_\beta)$
is of degree $d=d_\alpha+d_\beta$ in the sense of
Definition~\ref{def:covbin} (Definition~\ref{def:semibin}).
\end{rem}
The ring
$\bbC[V_4\oplus V_6\oplus\bbC^2]^{\grp{SL_2(\bbC)}}$
of covariants and the ring
$\bbC[V_4\oplus V_6]^{\grp{U}_2(\bbC)}$
of semiinvariants admit a trigrading,
graded by degrees $d_\alpha,d_\beta$ and order $\omega$.

We now move on to the ring $J^{E_8}_{*,*}$
of $W(E_8)$-invariant weak Jacobi forms.
We see that similar refined degrees can be introduced:
\begin{dfn}\label{def:refdegab}
We say that 
$\xi\in
\bbC[a_0,a_2,\ldots,a_4,b_0,\ldots,b_6,a_0^{-1},b_0^{-1}]$
is of degrees $(d_a,d_b)$
if $\xi$ is homogeneous of degree $d_a$ in $a_i$ and
homogeneous of degree $d_b$ in $b_j$.
\end{dfn}
Similarly, we can introduce refined degrees $(d_c,d_d)$
in terms of $c_i,d_j$ to
\begin{align}
 \bbC[c_0,\ldots,c_4,d_0,d_2,\ldots,d_6,c_0^{-1},d_0^{-1}]
=\bbC[a_0,a_2,\ldots,a_4,b_0,\ldots,b_6,a_0^{-1},b_0^{-1}].
\label{eq:abcdring}
\end{align}
Then the following fact is clear from \eqref{eq:abincd}:
\begin{prop}
The degrees $(d_c,d_d)$ are identical with $(d_a,d_b)$.
\end{prop}
Introduction of these degrees to the ring $J^{E_8}_{*,*}$
yields an interesting result:
\begin{thm}\label{thm:trideg}
$W(E_8)$-invariant weak Jacobi forms form
a trigraded ring over $\bbC$, graded by degrees $d_a,d_b$ and
index $m$
\begin{align}
J^{E_8}_{*,*}=\bigoplus_{d_a,d_b,m=0}^\infty J^{E_8}_{d_a,d_b,m}.
\end{align}
Every element of $J^{E_8}_{d_a,d_b,m}$ is of weight
\begin{align}
k=4d_a+6d_b-6m.
\label{eq:graderel}
\end{align}
\end{thm}
\begin{proof}
By Theorem~\ref{thm:JE8inR},
the grading by degrees $(d_a,d_b)$
is consistently introduced to the ring $J^{E_8}_{*,*}$:
any element of $J^{E_8}_{*,*}$ is expressed uniquely as
a polynomial of $a_i,b_j$
and degrees can be determined term-by-term. The functions
$a_i$ and $b_j$ are of degrees $(1,0)$ and $(0,1)$, respectively,
and have weight and index described
in Proposition~\ref{prop:abprop} (1).
Their degrees $(d_a,d_b)$, weight $k$
and index $m$ are subject to the relation \eqref{eq:graderel}.
\end{proof}
\begin{rem}
This grading looks somewhat unfamiliar. For instance,
$\Delta=a_0^3-27b_0^2$ is not homogeneous with respect to this grading.
 The holomorphic Jacobi form $A_1=-3a_0b_1$ is homogeneous,
but the other $A_m,B_m\ (m\ge 2)$ are not. 
\end{rem}
%

%%%
\subsection{Isomorphism}
%%%

In this subsection we prove another main theorem
of this paper:
Theorem~\ref{thm:JandS}, or equivalently Theorem~\ref{thm:JandCov}.
We do this by constructing an explicit isomorphism.

Let us define $\ha_i,\hb_i,\hc_i,\hd_i$ by the relations
\begin{align}
\begin{aligned}
\sum_{i=0}^4\ha_i u^{4-i}
&=\sum_{i=0}^4\alpha_i
 \left(u-\frac{\alpha_1}{4\alpha_0}\right)^{4-i},\quad&
\sum_{i=0}^6\hb_i u^{6-i}
&=\sum_{i=0}^6\beta_i
 \left(u-\frac{\alpha_1}{4\alpha_0}\right)^{6-i},\\
\sum_{i=0}^4\hc_i \tu^{4-i}
&=\sum_{i=0}^4\alpha_i
 \left(\tu-\frac{\beta_1}{6\beta_0}\right)^{4-i},&
\sum_{i=0}^6\hd_i \tu^{6-i}
&=\sum_{i=0}^6\beta_i
 \left(\tu-\frac{\beta_1}{6\beta_0}\right)^{6-i}.
\end{aligned}
\label{eq:habcddef}
\end{align}
Explicitly, they are given by
\begin{align}
\begin{aligned}
\ha_i&=\sum_{j=0}^i\alpha_j
 \begin{pmatrix} 4-j\\[1ex] 4-i\end{pmatrix}
 \left(-\frac{\alpha_1}{4\alpha_0}\right)^{i-j},\quad&&
\hb_i=\sum_{j=0}^i\beta_j
 \begin{pmatrix} 6-j\\[1ex] 6-i\end{pmatrix}
 \left(-\frac{\alpha_1}{4\alpha_0}\right)^{i-j},\\
\hc_i&=\sum_{j=0}^i\alpha_j
 \begin{pmatrix} 4-j\\[1ex] 4-i\end{pmatrix}
 \left(-\frac{\beta_1}{6\beta_0}\right)^{i-j},&&
\hd_i=\sum_{j=0}^i\beta_j
 \begin{pmatrix} 6-j\\[1ex] 6-i\end{pmatrix}
 \left(-\frac{\beta_1}{6\beta_0}\right)^{i-j}.
\end{aligned}
\label{eq:habcdform}
\end{align}
Note that
\begin{align}
\ha_1=0,\quad\hd_1=0.
\label{eq:ha1hd1}
\end{align}
\begin{lem}\label{lem:abcdrel}
$\ha_i,\hb_i,\hc_i,\hd_i$ satisfy
the same relations as
\eqref{eq:abincd} for $a_i,b_i,c_i,d_i$.
\end{lem}
\begin{proof}
Since $u$ and $\tu$ in \eqref{eq:habcddef} are
independent formal variables,
we can rewrite them in terms of a new variable $\breve{u}$ by
\begin{align}
u-\frac{\alpha_1}{4\alpha_0}=\tu-\frac{\beta_1}{6\beta_0}=\breve{u}.
\end{align}
This yields the relations
\begin{align}
\sum_{i=0}^4\ha_i u^{4-i}=\sum_{i=0}^4\hc_i \tu^{4-i},\qquad
\sum_{i=0}^6\hb_i u^{6-i}=\sum_{i=0}^6\hd_i \tu^{6-i}.
\end{align}
Combining these with \eqref{eq:ha1hd1},
we obtain the same relations as in \eqref{eq:abincd}.
\end{proof}
\begin{lem}\label{lem:abcdsemi}
$\alpha_0^{i-1}\ha_i,\ \alpha_0^i\hb_i,\ 
\beta_0^i\hc_i,\ \beta_0^{i-1}\hd_i$ are joint
semiinvariants of $f$ and $g$.
\end{lem}
\begin{proof}
This is a special case of Proposition~\ref{prop:alsemi}.
\end{proof}

By the substitution $a_i=\ha_i,b_i=\hb_i,c_i=\hc_i,d_i=\hd_i$
let us introduce the mappings
\begin{align}
\begin{aligned}
\psi&:
\cR=\bbC[a_0,a_2,a_3,a_4,b_0,b_1,b_2,b_3,b_4,b_5,b_6]
\to
\bbC[\alpha_0,\ldots,\alpha_4,\beta_0,\ldots,\beta_6,\alpha_0^{-1}],\\
\tpsi&:
\tcR=\bbC[c_0,c_1,c_2,c_3,c_4,d_0,d_2,d_3,d_4,d_5,d_6]
\to
\bbC\left[\alpha_0,\ldots,\alpha_4,
\beta_0,\ldots,\beta_6,\beta_0^{-1}\right].
\end{aligned}
\label{eq:symrecmap}
\end{align}
It is clear from \eqref{eq:habcdform} that
these mappings have the following properties.
\begin{lem}\label{lem:psiinj}
The mappings $\psi$ and $\tpsi$ are injective.
The inverse mappings from their images are
given by the substitutions
\begin{align}
\begin{aligned}
\psi^{-1}&:\ 
\alpha_i=a_i\ (i\ne 1),\quad \alpha_1=0,\quad\beta_j=b_j,\\
\tpsi^{-1}&:\ 
\alpha_i=c_i,\quad \beta_j=d_j\ (j\ne 1),\quad \beta_1=0.
\end{aligned}
\end{align}
These mappings identify the degrees $(d_a,d_b)$
with $(d_\alpha,d_\beta)$.
\end{lem}
Recall that $J^{E_8}_{*,*}=\cR\cap\tcR$
by Theorem~\ref{thm:intersec}.
\begin{lem}\label{lem:hpsimap}
The restrictions of $\psi$ and $\tpsi$ to
$J^{E_8}_{*,*}=\cR\cap\tcR$
are identical
\begin{align}
\hpsi:=
\psi\big|_{J^{E_8}_{*,*}}=\tpsi\big|_{J^{E_8}_{*,*}}
\end{align}
and define an injection
\begin{align}
\hpsi:J^{E_8}_{*,*}\to
\bbC[\alpha_0,\ldots,\alpha_4,\beta_0,\ldots,\beta_6].
\end{align}
This injection identifies
the degrees $(d_a,d_b)$ with $(d_\alpha,d_\beta)$.
\end{lem}
\begin{proof}
The restrictions
$\psi\big|_{J^{E_8}_{*,*}}$ and $\tpsi\big|_{J^{E_8}_{*,*}}$
are identical by Lemma~\ref{lem:abcdrel}.
The intersection of the codomains of $\psi$ and $\tpsi$
is $\bbC[\alpha_0,\ldots,\alpha_4,\beta_0,\ldots,\beta_6]$.
\end{proof}
We finally arrive at the main theorem of this paper.
\begin{thm}\label{thm:JandS}
The injection $\hpsi$ gives an isomorphism
between the ring of $W(E_8)$-invariant weak Jacobi forms 
and the ring of joint semiinvariants of $f$ and $g$:
\begin{align}
\begin{aligned}
J^{E_8}_{*,*}&\ \cong\ \bbC[V_4\oplus V_6]^{\grp{U}_2(\bbC)};\\
\hpsi:\varphi(a_i,b_j)&\ \mapsto\ \varphi(\ha_i,\hb_j),\\
\hpsi^{-1}:\Phi(a_i,b_j)&\ \mapsfrom\ \Phi(\alpha_i,\beta_j).
\end{aligned}
\end{align}
Every $W(E_8)$-invariant weak Jacobi form of
degrees $(d_a,d_b)$, weight $k$ 
and index $m$ is mapped
to a semiinvariant of the same degrees $(d_a,d_b)$
and order $\omega=k+4m$.
\end{thm}
\begin{proof}
Suppose that $\varphi(a_i,b_j)$ is a $W(E_8)$-invariant weak
Jacobi form of degrees $(d_a,d_b)$, weight $k$ and index $m$.
By Lemma~\ref{lem:hpsimap},
$\hpsi(\varphi)=\varphi(\ha_i,\hb_j)$ is
a polynomial that is homogeneous of degree $d_a$ in $\alpha_i$
and homogeneous of degree $d_b$ in $\beta_j$.
By Lemma~\ref{lem:abcdsemi}, $\varphi(\ha_i,\hb_j)$
is invariant under the transformation $(u',v')=(u+\kappa v,v)$.
To see the scaling behavior
under the transformation $(u',v')=(\lambda u,\lambda^{-1}v)$,
recall that $\ha_i,\hb_j$ scale as
$\ha'_i=\lambda^{4-2i}\ha_i$,
$\hb'_i=\lambda^{6-2i}\hb_i$
(see \eqref{eq:alorder}.)
This implies that $\varphi(\ha_i,\hb_j)$ scales as
\begin{align}
\varphi(\ha'_i,\hb'_j)=\lambda^{4d_a+6d_b-2m}\varphi(\ha_i,\hb_j).
\end{align}
(Recall that the subscript $i$ of $a_i,b_i$ represents their index.)
Therefore, $\varphi(\ha_i,\hb_j)$ is a semiinvariant of $f$ and $g$
of degrees $(d_a,d_b)$ and order
$\omega=4d_a+6d_b-2m$, which is also written as 
$\omega=k+4m$ by Theorem~\ref{thm:trideg}.

Conversely, suppose that $\Phi(\alpha_i,\beta_j)$ is
a joint semiinvariant of $f$ and $g$.
Then $\Phi$ is invariant under the translation \eqref{eq:u1trans}
and thus $\Phi(a_i,b_j)=\Phi(c_i,d_j)$.
Clearly, this is a polynomial of $a_i,b_j$ and,
at the same time, 
a polynomial of $c_i,d_j$.
By Theorem~\ref{thm:intersec},
this is a $W(E_8)$-invariant weak Jacobi form.
By Lemma~\ref{lem:psiinj}, this gives the inverse mapping
$\hpsi^{-1}(\Phi(\alpha_i,\beta_j))=\Phi(a_i,b_j)=\Phi(c_i,d_j)$.

It is obvious that the bijection preserves
the structure of the rings
and thus gives an isomorphism.
\end{proof}

By the Roberts isomorphism (Theorem~\ref{thm:Roberts}),
the ring of semiinvariants and that of covariants are isomorphic.
Hence, we can rephrase Theorem~\ref{thm:JandS} as follows:
\begin{thm}\label{thm:JandCov}
The composite of the mappings \eqref{eq:Robinv} and $\hpsi$
gives an isomorphism between
the ring of $W(E_8)$-invariant weak Jacobi forms 
and the ring of joint covariants of $f$ and $g$:
\begin{align}
J^{E_8}_{*,*}\ \cong\ 
\bbC[V_4\oplus V_6\oplus\bbC^2]^{\grp{SL_2(\bbC)}}.
\end{align}
\end{thm}
As a corollary of Theorem~\ref{thm:JandCov},
we obtain a proof of the lower bound conjecture of Sun and Wang
on the weight of $W(E_8)$-invariant weak Jacobi forms.
\begin{cor}[{Sun and Wang \cite[Conjecture 6.1]{Sun:2021ije}}]
\label{cor:lowerbound}
The weight $k$ of non-zero $W(E_8)$-invariant weak Jacobi forms
of index $m$ is not less than $-4m$.
\end{cor}
\begin{proof}
By definition there are no covariants of negative order.
By the isomorphism, this means that
$\omega=k+4m\ge 0$ for every non-zero element of $J^{E_8}_{*,*}$. 
\end{proof}
In Section~\ref{sec:algorithm} (and also in \cite{Sakai:2022taq})
we studied the structure of the graded subring 
$\Jlb=\bigoplus_{m=0}^\infty J^{E_8}_{-4m,m}$,
which lies on the bound of Sun and Wang.
Another corollary of Theorem~\ref{thm:JandCov} is the following:
\begin{cor}\label{cor:JlbInv}
The graded subring $\Jlb$ is isomorphic to
the ring of joint invariants of $f$ and $g$:
\begin{align}
\Jlb=\bigoplus_{m=0}^\infty J^{E_8}_{-4m,m}\ 
\cong\ 
\bbC[V_4\oplus V_6]^{\grp{SL_2(\bbC)}}.
\end{align}
\end{cor}
\begin{proof}
This is a direct consequence of the fact that
a covariant of order $\omega=0$ is an invariant.
\end{proof}
\begin{rem}\label{rem:grades}
We summarize the relations among various grades:
the weight $k$ and index $m$ of Jacobi forms,
the degree $d$ and order $\omega$ of semiinvariants
(covariants),
and the common degrees $(d_a,d_b)=(d_c,d_d)=(d_\alpha,d_\beta)$.
They satisfy
\begin{align}
\begin{aligned}
d=d_a+d_b,\qquad
\omega=k+4m=4d_a+6d_b-2m,
\end{aligned}
\end{align}
or conversely,
\begin{align}
\begin{aligned}
k=4d_a+6d_b-6m
 =-8d_a-12d_b+3\omega,\qquad
m=2d_a+3d_b-\frac{1}{2}\omega.
\end{aligned}
\end{align}
\end{rem}
%

%%%%%%%%%%%%%%%%%%%%%%%%%%%%%%%%%%%%%%%%%%%%%%%%%%%%%%%%%%%%%%%%%%%%%%%%
\section{Generators}
%%%%%%%%%%%%%%%%%%%%%%%%%%%%%%%%%%%%%%%%%%%%%%%%%%%%%%%%%%%%%%%%%%%%%%%%

In this section we discuss a minimal basis of the ring $J^{E_8}_{*,*}$
of $W(E_8)$-invariant weak Jacobi forms.

Let $f_k(u,v)\ (k=1,2)$ be binary forms of degree $n_k$.
For $i\in\bbZ_{\ge 0}$,
the $i$th transvectant of $f_1$ and $f_2$ is defined as
\begin{align}
\tv{f_1}{f_2}{i}
:=\frac{(n_1-i)!(n_2-i)!}{n_1!n_2!}
\sum_{j=0}^i(-1)^j\begin{pmatrix}i\\ j\end{pmatrix}
\frac{\partial^i f_1}{\partial u^{i-j}\partial v^j}
\frac{\partial^i f_2}{\partial u^j\partial v^{i-j}}.
\end{align}
In \cite{Olive:2014}, Olive computed a minimal basis of generators
for the ring 
of joint covariants of a binary quartic and a binary sextic.
The basis consists of 194 generators,
which are expressed as (nested) transvectants of
the quartic and the sextic.

Recall that the isomorphism established in the last section
is extremely simple:
For any joint covariant expressed in terms of
\begin{align}
f=\sum_{i=0}^4\alpha_iu^{4-i}v^i,\qquad g=\sum_{i=0}^6\beta_iu^{6-i}v^i,
\end{align}
the corresponding
$W(E_8)$-invariant weak Jacobi form is obtained by simply setting
\begin{align}
u=1,\quad v=0,\quad
\alpha_i=a_i,\quad
\beta_j=b_j,\quad
a_1=0.
\label{eq:isosubs}
\end{align}
Here $a_i,b_j$ are
the basic $W(E_8)$-invariant meromorphic Jacobi forms.
Hence, by abusing notation
we write the basis elements of $J^{E_8}_{*,*}$
as transvectants of $f$ and $g$,
with the understanding that the above substitution is performed.

Thus we immediately obtain
194 Jacobi forms from Olive's results.
The only nontrivial part is the correspondence of the grades
of covariants and Jacobi forms
summarized in Remark~\ref{rem:grades}.
For the study of the ring $J^{E_8}_{*,*}$ of Jacobi forms,
it is useful to rearrange the basis elements
by index $m$ and weight $k$. With this in mind,
we label the basis elements as
\begin{align}
\ggen{d_a}{d_b}{m}{\omega}{n}.
\end{align}
The second subscript $\omega=4d_a+6d_b-2m$
is redundant, but it is convenient to have it,
so that the weight of the element is
immediately calculated as $k=\omega-4m$.
When multiple generators occur at the same set of grades
$(d_a,d_b,m)$,
we further add an extra label $[n]\ (n=1,2,3,\ldots)$
as the third subscript to distinguish them.

The basis elements of $J^{E_8}_{*,*}$ are then expressed as follows:
\begin{align}
\gen{1}{0}{0}{4} &= f,&
\gen{0}{1}{0}{6} &= g,\nn\\
\gen{1}{1}{1}{8} &= \tv{f}{g}{1},&
\gen{2}{0}{2}{4} &= \tv{f}{f}{2},\nn\\
\gen{1}{1}{2}{6} &= \tv{f}{g}{2},&
\gen{0}{2}{2}{8} &= \tv{g}{g}{2},\nn\\
\gen{1}{1}{3}{4} &= \tv{f}{g}{3},&
\gen{3}{0}{3}{6} &= \tv{f}{\gen{2}{0}{2}{4}}{1},\nn\\
\gen{2}{1}{3}{8} &= \tv{\gen{2}{0}{2}{4}}{g}{1},&
\gen{1}{2}{3}{10} &= \tv{f}{\gen{0}{2}{2}{8}}{1},\nn\\
&\vdots&&\vdots\nn\\
\gen{3}{9}{33}{0}
 &= \tv{f\gen{2}{0}{2}{4}}{\gen{0}{3}{8}{2}\ggen{0}{6}{15}{6}{2}}{8},&
\gen{1}{11}{35}{0}
 &= \tv{f}{\gen{0}{3}{8}{2}\gen{0}{8}{23}{2}}{4},\nn\\
\gen{0}{12}{35}{2}
 &= \tv{\gen{0}{3}{5}{8}}{(\gen{0}{3}{8}{2})^3}{6},&
\gen{2}{11}{37}{0}
 &= \tv{f^2}{\gen{0}{3}{6}{6}\gen{0}{8}{23}{2}}{8},\nn\\
\gen{1}{13}{41}{0}
 &= \tv{f}{\gen{0}{3}{8}{2}\gen{0}{10}{29}{2}}{4},&
\gen{0}{15}{45}{0}
 &= \tv{\gen{0}{3}{5}{8}}{(\gen{0}{3}{8}{2})^4}{8}.
\end{align}
The complete list of 194 generators is presented
in Appendix~\ref{app:gen}.
By the substitution \eqref{eq:isosubs} we obtain,
for instance,
\begin{align}
\gen{1}{0}{0}{4} &= f=a_0,&
\gen{0}{1}{0}{6} &= g=b_0,\nn\\
\gen{1}{1}{1}{8} &= \tv{f}{g}{1}=-\frac{a_0b_1}{6},&
\gen{2}{0}{2}{4} &= \tv{f}{f}{2}=\frac{a_0a_2}{3},\nn\\
\gen{1}{1}{2}{6} &= \tv{f}{g}{2}=\frac{2a_0b_2+5a_2b_0}{30},&
\gen{0}{2}{2}{8} &= \tv{g}{g}{2}=\frac{12b_0b_2-5b_1^2}{90}.
\end{align}
\begin{table}[t]
\begin{align*}
\begin{array}{|c|ccccccc|c|}
\hline
m\backslash\omega&0&2&4&6&8&10&12&\#\\ \hline
 0& -& -& 1& 1& -& -& -&  2\\
 1& -& -& -& -& 1& -& -&  1\\
 2& -& -& 1& 1& 1& -& -&  3\\
 3& -& -& 1& 1& 1& 1& 1&  5\\
 4& 1& 1& 1& 1& 1& -& -&  5\\
 5& -& -& 2& 2& 1& -& -&  5\\
 6& 2& 2& 2& 1& -& -& -&  7\\
 7& -& 1& 2& 3& 2& 1& -&  9\\
 8& 1& 3& 3& 1& -& -& -&  8\\
 9& 1& 2& 4& 1& -& -& -&  8\\
10& 2& 3& 3& -& -& -& -&  8\\
11& -& 4& 5& 3& 1& -& -& 13\\
12& 3& 4& 2& -& -& -& -&  9\\
13& 1& 5& 3& -& -& -& -&  9\\
14& 3& 5& 1& -& -& -& -&  9\\
15& 3& 7& 3& 2& -& -& -& 15\\
16& 3& 4& -& -& -& -& -&  7\\
17& 3& 6& 1& -& -& -& -& 10\\
18& 4& 3& -& -& -& -& -&  7\\
19& 3& 4& 1& -& -& -& -&  8\\
20& 3& 1& -& -& -& -& -&  4\\
21& 4& 4& 1& -& -& -& -&  9\\
22& 2& 1& -& -& -& -& -&  3\\
\hline
\multicolumn{9}{c}{}
\end{array}
\hspace{1ex}
\begin{array}{|c|ccccccc|c|}
\hline
m\backslash\omega&0&2&4&6&8&10&12&\#\\ \hline
23& 3& 2& -& -& -& -& -&  5\\
24& 2& -& -& -& -& -& -&  2\\
25& 3& 2& 1& -& -& -& -&  6\\
26& 1& -& -& -& -& -& -&  1\\
27& 1& 1& -& -& -& -& -&  2\\
28& 1& -& -& -& -& -& -&  1\\
29& 2& 1& -& -& -& -& -&  3\\
30& 1& -& -& -& -& -& -&  1\\
31& 1& 1& -& -& -& -& -&  2\\
32& 1& -& -& -& -& -& -&  1\\
33& 1& -& -& -& -& -& -&  1\\
34& -& -& -& -& -& -& -&  0\\
35& 1& 1& -& -& -& -& -&  2\\
36& -& -& -& -& -& -& -&  0\\
37& 1& -& -& -& -& -& -&  1\\
38& -& -& -& -& -& -& -&  0\\
39& -& -& -& -& -& -& -&  0\\
40& -& -& -& -& -& -& -&  0\\
41& 1& -& -& -& -& -& -&  1\\
42& -& -& -& -& -& -& -&  0\\
43& -& -& -& -& -& -& -&  0\\
44& -& -& -& -& -& -& -&  0\\
45& 1& -& -& -& -& -& -&  1\\ \hline
\mbox{Tot.}&60&68&38&17& 8& 2& 1&194\\
\hline
\end{array}
\end{align*}
\caption{Number of generators of $J^{E_8}_{*,*}$,
arranged by index $m$ and order $\omega$.
The weight of generators is given by $k=\omega-4m$.
Observe that the numbers in the first column ($\omega=0$)
are in agreement with those in Table~\ref{table:dlbnew}.}
\label{table:numgen}
\end{table}

Table~\ref{table:numgen} shows the number
of generators arranged by index $m$ and order $\omega$.
The weight of generators is given by $k=\omega-4m$.
As we saw in Corollary~\ref{cor:JlbInv},
the graded subring $\Jlb=\bigoplus_{m=0}^\infty J^{E_8}_{-4m,m}$,
which we studied in Section~\ref{sec:algorithm},
is isomorphic to the ring of joint invariants of $f$ and $g$.
We now see that $\Jlb$ is generated by 60 generators of $\omega=0$.
In Section~\ref{sec:algorithm}, we obtained the number of
generators of index $m$ of $\Jlb$ for $m\le 32$,
as in Table~\ref{table:dlbnew}.
These numbers are in perfect agreement with
those in the first column ($\omega=0$)
of Table~\ref{table:numgen}, as expected.
This serves as a quite nontrivial check of the isomorphism.

%%%%%%%%%%%%%%%%%%%%%%%%%%%%%%%%%%%%%%%%%%%%%%%%%%%%%%%%%%%%%%%%%%%%%%%%
\vspace{2ex}

\begin{center}
  {\bf Acknowledgments}
\end{center}

The author is grateful to Haowu Wang for helpful comments.
This work was supported in part by JSPS KAKENHI Grant Number 19K03856.

\vspace{2ex}

%%%%%%%%%%%%%%%%%%%%%%%%%%%%%%%%%%%%%%%%%%%%%%%%%%%%%%%%%%%%%%%%%%%%%%%%
%%% Appendices %%%

\appendix
\renewcommand{\theequation}{\Alph{section}.\arabic{equation}}

\newpage
%%%%%%%%%%%%%%%%%%%%%%%%%%%%%%%%%%%%%%%%%%%%%%%%%%%%%%%%%%%%%%%%%%%%%%%%
\section{Basic meromorphic Jacobi forms}\label{app:coeff}
%%%%%%%%%%%%%%%%%%%%%%%%%%%%%%%%%%%%%%%%%%%%%%%%%%%%%%%%%%%%%%%%%%%%%%%%

%
\begin{align}
a_0&=\frac{E_4}{12},\qquad
a_1=0,\qquad
a_2=
 \frac{6}{E_4\Delta}\Bigl(-E_4A_2+A_1^2\Bigr),\nn\\
%\noalign{\break}
a_3&=\frac{1}{9E_4^2\Delta^2}
 \Bigl(-7E_4^2E_6A_3-20E_4^3B_3
       -9E_4E_6A_1A_2+30E_4^2A_1B_2+6E_6A_1^3\Bigr),\nn\\
%\noalign{\break}
a_4&=\frac{1}{288E_4^3\Delta^3}
 \Bigl(576E_4^3\Delta A_4+32256E_4^2\Delta A_1A_3
 -9E_4^5A_2^2-30E_4^3E_6A_2B_2\nn\\
&\hspace{1em}
 -25E_4^4B_2^2+(60E_4^4-12E_4E_6^2)A_1^2A_2
 +80E_4^2E_6A_1^2B_2+(-70E_4^3+6E_6^2)A_1^4\Bigr),\nn\\
%\noalign{\break}
b_0&=\frac{E_6}{216},\qquad
b_1=-\frac{4}{E_4}A_1,\qquad
b_2=
 \frac{5}{6E_4^2\Delta}\Bigl(E_4^2B_2-E_6A_1^2\Bigr),\nn\\
%\noalign{\break}
b_3&=\frac{1}{108E_4^3\Delta^2}
 \Bigl(-7E_4^5A_3-20E_4^3E_6B_3
 -9E_4^4A_1A_2+30E_4^2E_6A_1B_2\nn\\
&\hspace{1em}
 +(16E_4^3-10E_6^2)A_1^3\Bigr),\nn\\
%\noalign{\break}
b_4&=\frac{1}{1728E_4^4\Delta^3}
 \Bigl(-8640E_4^4\Delta B_4
 +138240E_4^3\Delta A_1B_3
 +9E_4^5E_6A_2^2+30E_4^6A_2B_2\nn\\
&\hspace{1em}
 +25E_4^4E_6B_2^2
 -48E_4^4E_6A_1^2A_2
 +(-140E_4^5+60E_4^2E_6^2)A_1^2B_2\nn\\
&\hspace{1em}
 +(74E_4^3E_6-10E_6^3)A_1^4\Bigr),\nn\\
%\noalign{\break}
b_5&=\frac{1}{72E_4^5\Delta^3}\Bigl(
 -36288E_4^4\Delta A_5
 -294E_4^6A_2A_3-770E_4^4E_6B_2A_3
 -840E_4^4E_6A_2B_3\nn\\
&\hspace{1em}
 -2200E_4^5B_2B_3+168E_4^5A_1^2A_3
 +480E_4^3E_6A_1^2B_3
 -621E_4^5A_1A_2^2+3525E_4^4A_1B_2^2
\nn\\
&\hspace{1em}
 +1224E_4^4A_1^3A_2-240E_4^2E_6A_1^3B_2
 +(-456E_4^3+24E_6^2)A_1^5
\Bigr),\nn\\
%\noalign{\break}
b_6&=\frac{1}{4478976E_4^6\Delta^5}\Bigl(
 -19906560E_4^6\Delta^2B_6
 -188116992E_4^4E_6\Delta^2A_1A_5\nn\\
&\hspace{1em}
 -5184E_4^7E_6\Delta A_2A_4
 -8640E_4^8\Delta B_2A_4
 -103680E_4^8\Delta A_2B_4\nn\\
&\hspace{1em}
 -172800E_4^6E_6\Delta B_2B_4
 +12672E_4^6E_6\Delta A_1^2A_4
 +17280(5E_4^7+9E_4^4E_6^2)\Delta A_1^2B_4\nn\\
&\hspace{1em}
 +112896E_4^7E_6\Delta A_3^2
 +645120E_4^8\Delta A_3B_3
 +921600E_4^6E_6\Delta B_3^2\nn\\
&\hspace{1em}
 -1717632E_4^6E_6\Delta A_1A_2A_3
 -362880(4E_4^7+11E_4^4E_6^2)\Delta A_1B_2A_3\nn\\
&\hspace{1em}
 +483840(4E_4^7-9E_4^4E_6^2)\Delta A_1A_2B_3
 -11404800E_4^5E_6\Delta A_1B_2B_3\nn\\
&\hspace{1em}
 +1161216E_4^5E_6\Delta A_1^3A_3
 -92160(37E_4^6-9E_4^3E_6^2)\Delta A_1^3B_3\nn\\
&\hspace{1em}
 +(135E_4^9E_6+54E_4^6E_6^3)A_2^3
 +(405E_4^{10}+540E_4^7E_6^2)A_2^2B_2
 +1575E_4^8E_6A_2B_2^2\nn\\
&\hspace{1em}
 +(375E_4^9+500E_4^6E_6^2)B_2^3
 +(-3159E_4^8E_6+1701E_4^5E_6^3)A_1^2A_2^2\nn\\
&\hspace{1em}
 +(-3060E_4^9-1800E_4^6E_6^2)A_1^2A_2B_2
 +(6975E_4^7E_6-11025E_4^4E_6^3)A_1^2B_2^2\nn\\
&\hspace{1em}
 +(6768E_4^7E_6-3024E_4^4E_6^3)A_1^4A_2
 +(4260E_4^8+1800E_4^5E_6^2+180E_4^2E_6^4)A_1^4B_2\nn\\
&\hspace{1em}
 +(-3692E_4^6E_6+504E_4^3E_6^3-12E_6^5)A_1^6
\Bigr).
\label{eq:abinABfull}
\end{align}
Here, $E_4,E_6$ are Eisenstein series \eqref{eq:eisendef},
$\Delta=(E_4^3-E_6^2)/1728$
and $A_i,B_j$ are $W(E_8)$-invariant holomorphic Jacobi
forms given in \eqref{E8AB}.
The above expressions are equivalent to those
in \cite[Appendix A]{Sakai:2011xg},
but expressed in a slightly more concise manner
by replacing $(E_4^3-E_6^2)$ with $1728\Delta$.
\begin{align}
c_0&=\frac{E_4}{12},\qquad
c_1=\frac{48A_1}{E_6},\qquad
c_2=
 \frac{6}{E_6^2\Delta}\Bigl(-E_6^2A_2+E_4^2A_1^2\Bigr),\nn\\
%\noalign{\break}
c_3&=\frac{1}{9E_6^3\Delta^2}
 \Bigl(-7E_6^4A_3-20E_4E_6^3B_3-9E_4^2E_6^2A_1A_2+30E_6^3A_1B_2\nn\\
&\hspace{1em}
 +(3E_4^4+3E_4E_6^2)A_1^3\Bigr),\nn\\
%\noalign{\break}
c_4&=\frac{1}{288E_6^4\Delta^3}
 \Bigl(576E_6^4\Delta A_4-92160E_6^3\Delta A_1B_3
 -9E_4^2E_6^4A_2^2-30E_6^5A_2B_2\nn\\
&\hspace{1em}
 -25E_4E_6^4B_2^2+(-12E_4^4E_6^2+60E_4E_6^4)A_1^2A_2
 +80E_4^2E_6^3A_1^2B_2\nn\\
&\hspace{1em}
 +(2E_4^6+6E_4^3E_6^2-72E_6^4)A_1^4\Bigr),\nn\\
%\noalign{\break}
d_0&=\frac{E_6}{216},\qquad
d_1=0,\qquad
d_2=
 \frac{5}{6E_6\Delta}\Bigl(E_6B_2-E_4A_1^2\Bigr),\nn\\
%\noalign{\break}
d_3&=\frac{1}{108E_6^2\Delta^2}
 \Bigl(-7E_4^2E_6^2A_3-20E_6^3B_3-9E_4E_6^2A_1A_2+30E_4^2E_6A_1B_2\nn\\
&\hspace{1em}
 +(-20E_4^3+26E_6^2)A_1^3\Bigr),\nn\\
%\noalign{\break}
d_4&=\frac{1}{1728E_6^3\Delta^3}
 \Bigl(-8640E_6^3\Delta B_4-48384E_4E_6^2\Delta A_1A_3
 +9E_4E_6^4A_2^2+30E_4^2E_6^3A_2B_2\nn\\
&\hspace{1em}
 +25E_6^4B_2^2
+(-36E_4^3E_6^2-12E_6^4)A_1^2A_2
+(60E_4^4E_6-140E_4E_6^3)A_1^2B_2\nn\\
&\hspace{1em}
+(-30E_4^5+94E_4^2E_6^2)A_1^4\Bigr),\nn\\
%\noalign{\break}
d_5&=\frac{1}{72E_6^4\Delta^3}\Bigl(
-21E_4^2E_6^4A_5
-60E_4^2E_6^3A_1B_4
-294E_4E_6^4A_2A_3
-2200E_6^4B_2B_3\nn\\
&\hspace{1em}
+(-168E_4^3E_6^2+336E_6^4)A_1^2A_3
+480E_4E_6^3A_1^2B_3
-513E_6^4A_1A_2^2\nn\\
&\hspace{1em}
+360E_4E_6^3A_1A_2B_2
-216E_4^2E_6^2A_1^3A_2
+(240E_4^3E_6-1440E_6^3)A_1^3B_2\nn\\
&\hspace{1em}
+(-96E_4^4+432E_4E_6^2)A_1^5
\Bigr)
 +\frac{1}{72\Delta^3}\frac{P_{16,5}}{E_4},\nn\\
%\noalign{\break}
d_6&=
\frac{1}{13436928E_6^5\Delta^5}\Bigl(
-59719680E_6^5\Delta^2B_6
-564350976E_4E_6^4\Delta^2A_1A_5\nn\\
&\hspace{1em}
-15552E_4E_6^6\Delta A_2A_4
-25920E_4^2E_6^5\Delta B_2A_4
-311040E_4^2E_6^5\Delta A_2B_4\nn\\
&\hspace{1em}
-518400E_6^6\Delta B_2B_4
+38016E_6^6\Delta A_1^2A_4
-51840(9E_4^4E_6^3-23E_4E_6^5)\Delta A_1^2B_4\nn\\
&\hspace{1em}
+338688E_4E_6^6\Delta A_3^2
+1935360E_4^2E_6^5\Delta A_3B_3
+2764800E_6^6\Delta B_3^2\nn\\
&\hspace{1em}
-72576(63E_4^3E_6^4+8E_6^6)\Delta A_1A_2A_3
-16329600E_4E_6^5\Delta A_1B_2A_3\nn\\
&\hspace{1em}
-7257600E_4E_6^5\Delta A_1A_2B_3
-34214400E_4^2E_6^4\Delta A_1B_2B_3\nn\\
&\hspace{1em}
-870912(E_4^5E_6^2-5E_4^2E_6^4)\Delta A_1^3A_3
+552960(9E_4^3E_6^3-23E_6^5)\Delta A_1^3B_3\nn\\
&\hspace{1em}
+(405E_4^3E_6^6+162E_6^8)A_2^3
+(1215E_4^4E_6^5+1620E_4E_6^7)A_2^2B_2
+4725E_4^2E_6^6A_2B_2^2\nn\\
&\hspace{1em}
+(1125E_4^3E_6^5+1500E_6^7)B_2^3
+(-5103E_4^5E_6^4+729E_4^2E_6^6)A_1^2A_2^2\nn\\
&\hspace{1em}
+(1620E_4^6E_6^3-12420E_4^3E_6^5-3780E_6^7)A_1^2A_2B_2\nn\\
&\hspace{1em}
+(33075E_4^4E_6^4-45225E_4E_6^6)A_1^2B_2^2\nn\\
&\hspace{1em}
+(-648E_4^7E_6^2+10368E_4^4E_6^4+1512E_4E_6^6)A_1^4A_2\nn\\
&\hspace{1em}
+(540E_4^8E_6-7560E_4^5E_6^3+25740E_4^2E_6^5)A_1^4B_2\nn\\
&\hspace{1em}
+(-180E_4^9+1512E_4^6E_6^2-3924E_4^3E_6^4-7008E_6^6)A_1^6
\Bigr).
\label{eq:cdinABfull}
\end{align}
Here, $P_{16,5}$ appearing in the expression of $d_5$
is the $W(E_8)$-invariant holomorphic Jacobi form
of weight $16$ and index $5$ given in \eqref{eq:P16c5}.
As explained in the main text,
$a_i,b_j$ and $c_k,d_l$ are related by \eqref{eq:abincd}.

%%%%%%%%%%%%%%%%%%%%%%%%%%%%%%%%%%%%%%%%%%%%%%%%%%%%%%%%%%%%%%%%%%%%%%%%
\section{Basic holomorphic Jacobi forms as polynomials}
%%%%%%%%%%%%%%%%%%%%%%%%%%%%%%%%%%%%%%%%%%%%%%%%%%%%%%%%%%%%%%%%%%%%%%%%

%
\begin{align}
A_1&=-3a_0b_1,\qquad
A_2=\frac{9a_0b_1^2-2\Delta a_2}{12},\nn\\
A_3&=\frac{-21a_0b_1^3-12\Delta a_0b_3-6\Delta a_2b_1
 +18\Delta a_3b_0}{112},\nn\\
A_4&=\frac{1}{64}(
 4\Delta a_0^2a_2^2+3a_0b_1^4-96\Delta a_0b_1b_3+48\Delta a_0b_2^2
 -144\Delta a_2b_0b_2+36\Delta a_2b_1^2\nn\\
&\hspace{2em}
 +144\Delta a_3b_0b_1+32\Delta^2a_4),\nn\\
A_5&=\frac{1}{5376}(
 36\Delta a_0^2a_2^2b_1-63a_0b_1^5+216\Delta a_0b_1^2b_3
 -144\Delta a_0b_1b_2^2-432\Delta a_2b_0b_1b_2\nn\\
&\hspace{2em}
 -100\Delta a_2b_1^3+1980\Delta a_3b_0b_1^2-128\Delta^2a_0b_5
 -112\Delta^2a_2b_3+176\Delta^2a_3b_2),\nn\\[1ex]
B_2&=\frac{135b_0b_1^2+12\Delta b_2}{10},\qquad
B_3=\frac{-3\Delta a_0^2a_3-270b_0b_1^3+54\Delta b_0b_3
 -36\Delta b_1b_2}{80},\nn\\
B_4&=\frac{1}{160}(
 -16\Delta a_0^2a_2b_2+24\Delta a_0^2a_3b_1+12\Delta a_0a_2^2b_0
 +135b_0b_1^4-432\Delta b_0b_1b_3\nn\\
&\hspace{2em}
 +144\Delta b_0b_2^2+24\Delta b_1^2b_2-32\Delta^2b_4),\nn\\
B_6&=\frac{1}{2560}(
 48\Delta a_0^3b_1^2b_4+48\Delta a_0^2a_2b_1^2b_2
 -144\Delta a_0^2a_3b_1^3-144\Delta a_0^2a_4b_0b_1^2\nn\\
&\hspace{2em}
 -108\Delta a_0a_2^2b_0b_1^2+135b_0b_1^6-64\Delta^2a_0^2a_2b_4
 +48\Delta^2a_0^2a_3b_3-96\Delta^2a_0^2a_4b_2\nn\\
&\hspace{2em}
 -8\Delta^2a_0a_2^2b_2+16\Delta^2a_0a_2a_3b_1+144\Delta^2a_0a_2a_4b_0
 -36\Delta^2a_0a_3^2b_0+12\Delta^2a_2^3b_0\nn\\
&\hspace{2em}
 +1296\Delta b_0^2b_1^2b_4-216\Delta b_0b_1^3b_3+36\Delta b_1^4b_2
 -2592\Delta^2b_0b_1b_5+1152\Delta^2b_0b_2b_4\nn\\
&\hspace{2em}
 -432\Delta^2b_0b_3^2-576\Delta^3b_6),\nn\\
\Delta&=a_0^3-27b_0^2.
\label{eq:ABinab}
\end{align}
\begin{align}
A_1&=\frac{9c_1d_0}{2},\qquad
A_2=\frac{3c_0^2c_1^2-8\Delta c_2}{48},\nn\\
A_3&=\frac{21c_0c_1^3d_0-96\Delta c_0d_3+96\Delta c_1d_2
 +144\Delta c_3d_0}{896},\nn\\
A_4&=\frac{1}{1024}(9c_1^4d_0^2-128\Delta c_0^2c_1c_3
 +64\Delta c_0^2c_2^2-16\Delta c_0c_1^2c_2
 +3\Delta c_1^4+768\Delta c_0d_2^2\nn\\
&\hspace{2em}
 +2304\Delta c_1d_0d_3-2304\Delta c_2d_0d_2+512\Delta^2c_4),\nn\\
A_5&=\frac{1}{172032}(21c_0^2c_1^5d_0+576\Delta c_0^2c_1^2d_3
 +768\Delta c_0^2c_1c_2d_2
 -384\Delta c_0c_1^3d_2+5280\Delta c_0c_1^2c_3d_0\nn\\
&\hspace{2em}
 -1728\Delta c_0c_1c_2^2d_0-224\Delta c_1^3c_2d_0
 +6912\Delta c_1d_0d_2^2-4096\Delta^2c_0d_5+2048\Delta^2c_1d_4\nn\\
&\hspace{2em}
 -3584\Delta^2c_2d_3+5632\Delta^2c_3d_2),\nn\\[1ex]
B_2&=\frac{45c_0c_1^2d_0+48\Delta d_2}{40},\nn\\
B_3&=\frac{270c_1^3d_0^2-24\Delta c_0^2c_3+12\Delta c_0c_1c_2
 -3\Delta c_1^3+432\Delta d_0d_3}{640},\nn\\
B_4&=\frac{1}{2560}(15c_0^2c_1^4d_0+384\Delta c_0^2c_1d_3
 -256\Delta c_0^2c_2d_2
 -96\Delta c_0c_1^2d_2-576\Delta c_0c_1c_3d_0\nn\\
&\hspace{2em}
 +192\Delta c_0c_2^2d_0-96\Delta c_1^2c_2d_0
 +2304\Delta d_0d_2^2-512\Delta^2d_4),\nn\\
B_6&=\frac{1}{163840}(135c_1^6d_0^3+48\Delta c_0^2c_1^4d_2
 +1344\Delta c_0^2c_1^3c_3d_0
 -576\Delta c_0^2c_1^2c_2^2d_0+48\Delta c_0c_1^4c_2d_0\nn\\
&\hspace{2em}
 -3\Delta c_1^6d_0+13824\Delta c_0c_1^2d_0^2d_4
 -8640\Delta c_1^3d_0^2d_3+6912\Delta c_1^2c_2d_0^2d_2\nn\\
&\hspace{2em}
 -20736\Delta c_1^2c_4d_0^3+9216\Delta^2c_0^2c_1d_5
 -4096\Delta^2c_0^2c_2d_4+3072\Delta^2c_0^2c_3d_3\nn\\
&\hspace{2em}
-6144\Delta^2c_0^2c_4d_2-768\Delta^2c_0c_1^2d_4
 +1536\Delta^2c_0c_1c_2d_3-1536\Delta^2c_0c_1c_3d_2\nn\\
&\hspace{2em}
 -512\Delta^2c_0c_2^2d_2+9216\Delta^2c_0c_2c_4d_0
 -2304\Delta^2c_0c_3^2d_0-192\Delta^2c_1^3d_3\nn\\
&\hspace{2em}
 -9216\Delta^2c_1^2c_4d_0-1536\Delta^2c_1c_2c_3d_0+768\Delta^2c_2^3d_0
 +73728\Delta^2d_0d_2d_4\nn\\
&\hspace{2em}-27648\Delta^2d_0d_3^2-36864\Delta^3d_6),\nn\\
\Delta&=c_0^3-27d_0^2.
\label{eq:ABincd}
\end{align}
The relations \eqref{eq:ABinab} and \eqref{eq:ABincd}
are obtained by inverting the relations
\eqref{eq:abinABfull} and \eqref{eq:cdinABfull},
respectively.

%%%%%%%%%%%%%%%%%%%%%%%%%%%%%%%%%%%%%%%%%%%%%%%%%%%%%%%%%%%%%%%%%%%%%%%%
\section{Full list of generators}\label{app:gen}
%%%%%%%%%%%%%%%%%%%%%%%%%%%%%%%%%%%%%%%%%%%%%%%%%%%%%%%%%%%%%%%%%%%%%%%%

%
\begin{align}
\gen{1}{0}{0}{4} &= f,&
\gen{0}{1}{0}{6} &= g,\nn\\
\gen{1}{1}{1}{8} &= \tv{f}{g}{1},&
\gen{2}{0}{2}{4} &= \tv{f}{f}{2},\nn\\
\gen{1}{1}{2}{6} &= \tv{f}{g}{2},&
\gen{0}{2}{2}{8} &= \tv{g}{g}{2},\nn\\
\gen{1}{1}{3}{4} &= \tv{f}{g}{3},&
\gen{3}{0}{3}{6} &= \tv{f}{\gen{2}{0}{2}{4}}{1},\nn\\
\gen{2}{1}{3}{8} &= \tv{\gen{2}{0}{2}{4}}{g}{1},&
\gen{1}{2}{3}{10} &= \tv{f}{\gen{0}{2}{2}{8}}{1},\nn\\
\gen{0}{3}{3}{12} &= \tv{g}{\gen{0}{2}{2}{8}}{1},&
\gen{2}{0}{4}{0} &= \tv{f}{f}{4},\nn\\
\gen{1}{1}{4}{2} &= \tv{f}{g}{4},&
\gen{0}{2}{4}{4} &= \tv{g}{g}{4},\nn\\
\gen{2}{1}{4}{6} &= \tv{\gen{2}{0}{2}{4}}{g}{2},&
\gen{1}{2}{4}{8} &= \tv{f}{\gen{0}{2}{2}{8}}{2},\nn\\
\ggen{2}{1}{5}{4}{1} &= \tv{f^2}{g}{5},&
\ggen{2}{1}{5}{4}{2} &= \tv{\gen{2}{0}{2}{4}}{g}{3},\nn\\
\ggen{1}{2}{5}{6}{1} &= \tv{f}{\gen{0}{2}{4}{4}}{1},&
\ggen{1}{2}{5}{6}{2} &= \tv{f}{\gen{0}{2}{2}{8}}{3},\nn\\
\gen{0}{3}{5}{8} &= \tv{g}{\gen{0}{2}{4}{4}}{1},&
\gen{3}{0}{6}{0} &= \tv{f}{\gen{2}{0}{2}{4}}{4},\nn\\
\gen{0}{2}{6}{0} &= \tv{g}{g}{6},&
\ggen{2}{1}{6}{2}{1} &= \tv{\gen{2}{0}{2}{4}}{g}{4},\nn\\
\ggen{2}{1}{6}{2}{2} &= \tv{f^2}{g}{6},&
\ggen{1}{2}{6}{4}{1} &= \tv{f}{\gen{0}{2}{4}{4}}{2},\nn\\
\ggen{1}{2}{6}{4}{2} &= \tv{f}{\gen{0}{2}{2}{8}}{4},&
\gen{0}{3}{6}{6} &= \tv{g}{\gen{0}{2}{4}{4}}{2},\nn\\
\gen{1}{2}{7}{2} &= \tv{f}{\gen{0}{2}{4}{4}}{3},&
\ggen{3}{1}{7}{4}{1} &= \tv{\gen{3}{0}{3}{6}}{g}{4},\nn\\
\ggen{3}{1}{7}{4}{2} &= \tv{f\gen{2}{0}{2}{4}}{g}{5},&
\ggen{2}{2}{7}{6}{1} &= \tv{\gen{2}{0}{2}{4}}{\gen{0}{2}{2}{8}}{3},\nn\\
\ggen{2}{2}{7}{6}{2} &= \tv{f^2}{\gen{0}{2}{2}{8}}{5},&
\ggen{2}{2}{7}{6}{3} &= \tv{\gen{2}{0}{2}{4}}{\gen{0}{2}{4}{4}}{1},\nn\\
\ggen{1}{3}{7}{8}{1} &= \tv{f}{\gen{0}{3}{5}{8}}{2},&
\ggen{1}{3}{7}{8}{2} &= \tv{f}{\gen{0}{3}{3}{12}}{4},\nn\\
\gen{0}{4}{7}{10} &= \tv{\gen{0}{2}{2}{8}}{\gen{0}{2}{4}{4}}{1},&
\gen{1}{2}{8}{0} &= \tv{f}{\gen{0}{2}{4}{4}}{4},\nn\\
\ggen{3}{1}{8}{2}{1} &= \tv{f\gen{2}{0}{2}{4}}{g}{6},&
\ggen{3}{1}{8}{2}{2} &= \tv{\gen{3}{0}{3}{6}}{g}{5},\nn\\
\gen{0}{3}{8}{2} &= \tv{g}{\gen{0}{2}{4}{4}}{4},&
\ggen{2}{2}{8}{4}{1} &= \tv{\gen{2}{0}{2}{4}}{\gen{0}{2}{2}{8}}{4},\nn\\
\ggen{2}{2}{8}{4}{2} &= \tv{\gen{2}{0}{2}{4}}{\gen{0}{2}{4}{4}}{2},&
\ggen{2}{2}{8}{4}{3} &= \tv{f^2}{\gen{0}{2}{2}{8}}{6},\nn\\
\gen{1}{3}{8}{6} &= \tv{f}{\gen{0}{3}{5}{8}}{3},&
\gen{3}{1}{9}{0} &= \tv{\gen{3}{0}{3}{6}}{g}{6},\nn\\
\ggen{2}{2}{9}{2}{1} &= \tv{\gen{2}{0}{2}{4}}{\gen{0}{2}{4}{4}}{3},&
\ggen{2}{2}{9}{2}{2} &= \tv{f^2}{\gen{0}{2}{2}{8}}{7},\nn\\
\gen{4}{1}{9}{4} &= \tv{(\gen{2}{0}{2}{4})^2}{g}{5},&
\ggen{1}{3}{9}{4}{1} &= \tv{f}{\gen{0}{3}{5}{8}}{4},\nn\\
\ggen{1}{3}{9}{4}{2} &= \tv{f}{\gen{0}{3}{8}{2}}{1},&
\ggen{1}{3}{9}{4}{3} &= \tv{f}{\gen{0}{3}{6}{6}}{3},\nn\\
\gen{0}{4}{9}{6} &= \tv{g}{\gen{0}{3}{8}{2}}{1},&
\ggen{2}{2}{10}{0}{1} &= \tv{f^2}{\gen{0}{2}{2}{8}}{8},\nn\\
\ggen{2}{2}{10}{0}{2} &= \tv{\gen{2}{0}{2}{4}}{\gen{0}{2}{4}{4}}{4},&
\gen{4}{1}{10}{2} &= \tv{(\gen{2}{0}{2}{4})^2}{g}{6},\nn\\
\ggen{1}{3}{10}{2}{1} &= \tv{f}{\gen{0}{3}{8}{2}}{2},&
\ggen{1}{3}{10}{2}{2} &= \tv{f}{\gen{0}{3}{6}{6}}{4},\nn\\
\ggen{3}{2}{10}{4}{1} &= \tv{f\gen{2}{0}{2}{4}}{\gen{0}{2}{2}{8}}{6},&
\ggen{3}{2}{10}{4}{2}
 &= \tv{\gen{3}{0}{3}{6}}{\gen{0}{2}{2}{8}}{5},\nn\\
\gen{0}{4}{10}{4} &= \tv{g}{\gen{0}{3}{8}{2}}{2},&
\ggen{3}{2}{11}{2}{1}
 &= \tv{f\gen{2}{0}{2}{4}}{\gen{0}{2}{2}{8}}{7},\nn\\
\ggen{3}{2}{11}{2}{2} &= \tv{f^3}{g^2}{11},&
\ggen{3}{2}{11}{2}{3}
 &= \tv{\gen{3}{0}{3}{6}}{\gen{0}{2}{2}{8}}{6},\nn\\
\ggen{3}{2}{11}{2}{4} &= \tv{\gen{3}{0}{3}{6}}{\gen{0}{2}{4}{4}}{4},&
\ggen{2}{3}{11}{4}{1} &= \tv{f^2}{\gen{0}{3}{3}{12}}{8},\nn\\
\ggen{2}{3}{11}{4}{2} &= \tv{\gen{2}{0}{2}{4}}{\gen{0}{3}{8}{2}}{1},&
\ggen{2}{3}{11}{4}{3}
 &= \tv{\gen{2}{0}{2}{4}}{\gen{0}{3}{6}{6}}{3},\nn\\
\ggen{2}{3}{11}{4}{4} &= \tv{\gen{2}{0}{2}{4}}{\gen{0}{3}{5}{8}}{4},&
\ggen{2}{3}{11}{4}{5} &= \tv{f^2}{\gen{0}{3}{6}{6}}{5},\nn\\
\ggen{1}{4}{11}{6}{1} &= \tv{f}{\gen{0}{4}{7}{10}}{4},&
\ggen{1}{4}{11}{6}{2} &= \tv{f}{\gen{0}{4}{10}{4}}{1},\nn\\
\ggen{1}{4}{11}{6}{3} &= \tv{f}{\gen{0}{4}{9}{6}}{2},&
\gen{0}{5}{11}{8} &= \tv{\gen{0}{2}{2}{8}}{\gen{0}{3}{8}{2}}{1},\nn\\
\ggen{3}{2}{12}{0}{1} &= \tv{f\gen{2}{0}{2}{4}}{\gen{0}{2}{2}{8}}{8},&
\ggen{3}{2}{12}{0}{2} &= \tv{f^3}{g^2}{12},\nn\\
\gen{0}{4}{12}{0} &= \tv{\gen{0}{2}{4}{4}}{\gen{0}{2}{4}{4}}{4},&
\ggen{2}{3}{12}{2}{1}
 &= \tv{\gen{2}{0}{2}{4}}{\gen{0}{3}{6}{6}}{4},\nn\\
\ggen{2}{3}{12}{2}{2} &= \tv{f^2}{\gen{0}{3}{6}{6}}{6},&
\ggen{2}{3}{12}{2}{3}
 &= \tv{\gen{2}{0}{2}{4}}{\gen{0}{3}{8}{2}}{2},\nn\\
\ggen{2}{3}{12}{2}{4} &= \tv{f^2}{\gen{0}{3}{5}{8}}{7},&
\ggen{1}{4}{12}{4}{1} &= \tv{f}{\gen{0}{4}{9}{6}}{3},\nn\\
\ggen{1}{4}{12}{4}{2} &= \tv{f}{\gen{0}{4}{10}{4}}{2},&
\gen{2}{3}{13}{0} &= \tv{f^2}{\gen{0}{3}{5}{8}}{8},\nn\\
\ggen{4}{2}{13}{2}{1} &= \tv{f^2\gen{2}{0}{2}{4}}{g^2}{11},&
\ggen{4}{2}{13}{2}{2}
 &= \tv{f\gen{3}{0}{3}{6}}{\gen{0}{2}{2}{8}}{8},\nn\\
\ggen{4}{2}{13}{2}{3}
 &= \tv{(\gen{2}{0}{2}{4})^2}{\gen{0}{2}{2}{8}}{7},&
\ggen{1}{4}{13}{2}{1} &= \tv{f}{\gen{0}{4}{10}{4}}{3},\nn\\
\ggen{1}{4}{13}{2}{2} &= \tv{f}{\gen{0}{4}{9}{6}}{4},&
\ggen{3}{3}{13}{4}{1}
 &= \tv{f\gen{2}{0}{2}{4}}{\gen{0}{3}{3}{12}}{8},\nn\\
\ggen{3}{3}{13}{4}{2} &= \tv{\gen{3}{0}{3}{6}}{\gen{0}{3}{6}{6}}{4},&
\gen{0}{5}{13}{4} &= \tv{\gen{0}{2}{4}{4}}{\gen{0}{3}{8}{2}}{1},\nn\\
\ggen{4}{2}{14}{0}{1} &= \tv{f^2\gen{2}{0}{2}{4}}{g^2}{12},&
\ggen{4}{2}{14}{0}{2}
 &= \tv{(\gen{2}{0}{2}{4})^2}{\gen{0}{2}{2}{8}}{8},\nn\\
\gen{1}{4}{14}{0} &= \tv{f}{\gen{0}{4}{10}{4}}{4},&
\ggen{3}{3}{14}{2}{1} &= \tv{f^3}{\gen{0}{3}{3}{12}}{11},\nn\\
\ggen{3}{3}{14}{2}{2} &= \tv{f\gen{2}{0}{2}{4}}{\gen{0}{3}{6}{6}}{6},&
\ggen{3}{3}{14}{2}{3}
 &= \tv{f\gen{2}{0}{2}{4}}{\gen{0}{3}{5}{8}}{7},\nn\\
\ggen{3}{3}{14}{2}{4} &= \tv{\gen{3}{0}{3}{6}}{\gen{0}{3}{5}{8}}{6},&
\gen{0}{5}{14}{2} &= \tv{\gen{0}{2}{4}{4}}{\gen{0}{3}{8}{2}}{2},\nn\\
\gen{2}{4}{14}{4} &= \tv{\gen{2}{0}{2}{4}}{\gen{0}{4}{9}{6}}{3},&
\ggen{3}{3}{15}{0}{1}
 &= \tv{f\gen{2}{0}{2}{4}}{\gen{0}{3}{5}{8}}{8},\nn\\
\ggen{3}{3}{15}{0}{2} &= \tv{\gen{3}{0}{3}{6}}{\gen{0}{3}{6}{6}}{6},&
\ggen{3}{3}{15}{0}{3} &= \tv{f^3}{\gen{0}{3}{3}{12}}{12},\nn\\
\ggen{5}{2}{15}{2}{1}
 &= \tv{\gen{2}{0}{2}{4}\gen{3}{0}{3}{6}}{\gen{0}{2}{2}{8}}{8},&
\ggen{5}{2}{15}{2}{2} &= \tv{f(\gen{2}{0}{2}{4})^2}{g^2}{11},\nn\\
\ggen{2}{4}{15}{2}{1} &= \tv{f^2}{g\gen{0}{3}{8}{2}}{7},&
\ggen{2}{4}{15}{2}{2}
 &= \tv{\gen{2}{0}{2}{4}}{\gen{0}{4}{10}{4}}{3},\nn\\
\ggen{2}{4}{15}{2}{3} &= \tv{f^2}{\gen{0}{4}{7}{10}}{8},&
\ggen{2}{4}{15}{2}{4} &= \tv{f^2}{\gen{0}{4}{9}{6}}{6},\nn\\
\ggen{2}{4}{15}{2}{5} &= \tv{\gen{2}{0}{2}{4}}{\gen{0}{4}{9}{6}}{4},&
\ggen{1}{5}{15}{4}{1} &= \tv{f}{\gen{0}{5}{14}{2}}{1},\nn\\
\ggen{1}{5}{15}{4}{2} &= \tv{f}{\gen{0}{5}{11}{8}}{4},&
\ggen{1}{5}{15}{4}{3} &= \tv{f}{\gen{0}{5}{13}{4}}{2},\nn\\
\ggen{0}{6}{15}{6}{1} &= \tv{\gen{0}{3}{5}{8}}{\gen{0}{3}{8}{2}}{2},&
\ggen{0}{6}{15}{6}{2}
 &= \tv{\gen{0}{3}{6}{6}}{\gen{0}{3}{8}{2}}{1},\nn\\
\gen{5}{2}{16}{0} &= \tv{f(\gen{2}{0}{2}{4})^2}{g^2}{12},&
\ggen{2}{4}{16}{0}{1}
 &= \tv{\gen{2}{0}{2}{4}}{\gen{0}{4}{10}{4}}{4},\nn\\
\ggen{2}{4}{16}{0}{2} &= \tv{f^2}{g\gen{0}{3}{8}{2}}{8},&
\ggen{4}{3}{16}{2}{1}
 &= \tv{f^2\gen{2}{0}{2}{4}}{\gen{0}{3}{3}{12}}{11},\nn\\
\ggen{4}{3}{16}{2}{2} &= \tv{f\gen{3}{0}{3}{6}}{\gen{0}{3}{5}{8}}{8},&
\ggen{1}{5}{16}{2}{1} &= \tv{f}{\gen{0}{5}{14}{2}}{2},\nn\\
\ggen{1}{5}{16}{2}{2} &= \tv{f}{\gen{0}{5}{13}{4}}{3},&
\ggen{4}{3}{17}{0}{1}
 &= \tv{f^2\gen{2}{0}{2}{4}}{\gen{0}{3}{3}{12}}{12},\nn\\
\ggen{4}{3}{17}{0}{2}
 &= \tv{(\gen{2}{0}{2}{4})^2}{\gen{0}{3}{5}{8}}{8},&
\gen{1}{5}{17}{0} &= \tv{f}{\gen{0}{5}{13}{4}}{4},\nn\\
\ggen{3}{4}{17}{2}{1} &= \tv{f^3}{\gen{0}{4}{7}{10}}{10},&
\ggen{3}{4}{17}{2}{2} &= \tv{f^3}{g\gen{0}{3}{6}{6}}{11},\nn\\
\ggen{3}{4}{17}{2}{3}
 &= \tv{\gen{3}{0}{3}{6}}{(\gen{0}{2}{4}{4})^2}{6},&
\ggen{3}{4}{17}{2}{4}
 &= \tv{f\gen{2}{0}{2}{4}}{\gen{0}{4}{9}{6}}{6},\nn\\
\ggen{3}{4}{17}{2}{5} &= \tv{f\gen{2}{0}{2}{4}}{\gen{0}{4}{7}{10}}{8},&
\ggen{3}{4}{17}{2}{6}
 &= \tv{\gen{3}{0}{3}{6}}{\gen{0}{4}{9}{6}}{5},\nn\\
\gen{2}{5}{17}{4} &= \tv{\gen{2}{0}{2}{4}}{\gen{0}{5}{11}{8}}{4},&
\ggen{3}{4}{18}{0}{1} &= \tv{f^3}{g\gen{0}{3}{6}{6}}{12},\nn\\
\ggen{3}{4}{18}{0}{2} &= \tv{f\gen{2}{0}{2}{4}}{g\gen{0}{3}{8}{2}}{8},&
\ggen{3}{4}{18}{0}{3}
 &= \tv{\gen{3}{0}{3}{6}}{\gen{0}{4}{9}{6}}{6},\nn\\
\gen{0}{6}{18}{0} &= \tv{\gen{0}{3}{8}{2}}{\gen{0}{3}{8}{2}}{2},&
\ggen{2}{5}{18}{2}{1}
 &= \tv{\gen{2}{0}{2}{4}}{\gen{0}{5}{14}{2}}{2},\nn\\
\ggen{2}{5}{18}{2}{2} &= \tv{f^2}{\gen{0}{5}{11}{8}}{7},&
\ggen{2}{5}{18}{2}{3}
 &= \tv{\gen{2}{0}{2}{4}}{\gen{0}{5}{13}{4}}{3},\nn\\
\gen{5}{3}{19}{0} &= \tv{f(\gen{2}{0}{2}{4})^2}{\gen{0}{3}{3}{12}}{12},&
\ggen{2}{5}{19}{0}{1}
 &= \tv{\gen{2}{0}{2}{4}}{\gen{0}{5}{13}{4}}{4},\nn\\
\ggen{2}{5}{19}{0}{2} &= \tv{f^2}{\gen{0}{5}{11}{8}}{8},&
\gen{4}{4}{19}{2}
 &= \tv{(\gen{2}{0}{2}{4})^2}{\gen{0}{4}{7}{10}}{8},\nn\\
\ggen{1}{6}{19}{2}{1} &= \tv{f}{\ggen{0}{6}{15}{6}{2}}{4},&
\ggen{1}{6}{19}{2}{2} &= \tv{f}{\ggen{0}{6}{15}{6}{1}}{4},\nn\\
\ggen{1}{6}{19}{2}{3} &= \tv{f}{(\gen{0}{3}{8}{2})^2}{3},&
\gen{0}{7}{19}{4} &= \tv{g}{(\gen{0}{3}{8}{2})^2}{3},\nn\\
\ggen{4}{4}{20}{0}{1}
 &= \tv{f^2\gen{2}{0}{2}{4}}{g\gen{0}{3}{6}{6}}{12},&
\ggen{4}{4}{20}{0}{2}
 &= \tv{(\gen{2}{0}{2}{4})^2}{g\gen{0}{3}{8}{2}}{8},\nn\\
\gen{1}{6}{20}{0} &= \tv{f}{(\gen{0}{3}{8}{2})^2}{4},&
\gen{0}{7}{20}{2} &= \tv{g}{(\gen{0}{3}{8}{2})^2}{4},\nn\\
\gen{6}{3}{21}{0} &= \tv{(\gen{2}{0}{2}{4})^3}{\gen{0}{3}{3}{12}}{12},&
\ggen{3}{5}{21}{0}{1}
 &= \tv{\gen{3}{0}{3}{6}}{\gen{0}{2}{4}{4}\gen{0}{3}{8}{2}}{6},\nn\\
\ggen{3}{5}{21}{0}{2} &= \tv{f^3}{g\gen{0}{4}{9}{6}}{12},&
\ggen{3}{5}{21}{0}{3}
 &= \tv{f\gen{2}{0}{2}{4}}{\gen{0}{5}{11}{8}}{8},\nn\\
\ggen{2}{6}{21}{2}{1} &= \tv{f^2}{\gen{0}{3}{6}{6}\gen{0}{3}{8}{2}}{7},&
\ggen{2}{6}{21}{2}{2}
 &= \tv{\gen{2}{0}{2}{4}}{(\gen{0}{3}{8}{2})^2}{3},\nn\\
\ggen{2}{6}{21}{2}{3}
 &= \tv{\gen{2}{0}{2}{4}}{\ggen{0}{6}{15}{6}{1}}{4},&
\ggen{2}{6}{21}{2}{4} &= \tv{f^2}{\ggen{0}{6}{15}{6}{2}}{6},\nn\\
\gen{1}{7}{21}{4} &= \tv{f}{\gen{0}{7}{19}{4}}{2},&
\ggen{2}{6}{22}{0}{1}
 &= \tv{f^2}{\gen{0}{3}{6}{6}\gen{0}{3}{8}{2}}{8},\nn\\
\ggen{2}{6}{22}{0}{2}
 &= \tv{\gen{2}{0}{2}{4}}{(\gen{0}{3}{8}{2})^2}{4},&
\gen{1}{7}{22}{2} &= \tv{f}{\gen{0}{7}{20}{2}}{2},\nn\\
\ggen{4}{5}{23}{0}{1}
 &= \tv{f\gen{3}{0}{3}{6}}{\gen{0}{2}{2}{8}\gen{0}{3}{8}{2}}{10},&
\ggen{4}{5}{23}{0}{2}
 &= \tv{f^2\gen{2}{0}{2}{4}}{g\gen{0}{4}{9}{6}}{12},\nn\\
\gen{1}{7}{23}{0} &= \tv{f}{\gen{0}{7}{19}{4}}{4},&
\gen{3}{6}{23}{2}
 &= \tv{\gen{3}{0}{3}{6}}{(\gen{0}{3}{8}{2})^2}{4},\nn\\
\gen{0}{8}{23}{2} &= \tv{\gen{0}{2}{4}{4}}{(\gen{0}{3}{8}{2})^2}{3},&
\ggen{3}{6}{24}{0}{1}
 &= \tv{f\gen{2}{0}{2}{4}}{\gen{0}{5}{14}{2}g}{8},\nn\\
\ggen{3}{6}{24}{0}{2} &= \tv{f^3}{(\gen{0}{3}{6}{6})^2}{12},&
\gen{5}{5}{25}{0} &= \tv{\gen{2}{0}{2}{4}\gen{3}{0}{3}{6}
 }{\gen{0}{2}{2}{8}\gen{0}{3}{8}{2}}{10},\nn\\
\ggen{2}{7}{25}{0}{1} &= \tv{\gen{2}{0}{2}{4}}{\gen{0}{7}{19}{4}}{4},&
\ggen{2}{7}{25}{0}{2}
 &= \tv{f^2}{\gen{0}{3}{8}{2}\gen{0}{4}{9}{6}}{8},\nn\\
\ggen{1}{8}{25}{2}{1} &= \tv{f}{\gen{0}{8}{23}{2}}{2},&
\ggen{1}{8}{25}{2}{2}
 &= \tv{f}{\gen{0}{3}{8}{2}\gen{0}{5}{14}{2}}{3},\nn\\
\gen{0}{9}{25}{4} &= \tv{\gen{0}{3}{5}{8}}{(\gen{0}{3}{8}{2})^2}{4},&
\gen{1}{8}{26}{0} &= \tv{f}{\gen{0}{3}{8}{2}\gen{0}{5}{14}{2}}{4},\nn\\
\gen{3}{7}{27}{0} &= \tv{f^3}{g\ggen{0}{6}{15}{6}{1}}{12},&
\gen{2}{8}{27}{2}
 &= \tv{f^2}{\gen{0}{3}{6}{6}\gen{0}{5}{14}{2}}{7},\nn\\
\gen{2}{8}{28}{0} &= \tv{f^2}{\gen{0}{3}{6}{6}\gen{0}{5}{14}{2}}{8},&
\gen{4}{7}{29}{0}
 &= \tv{f^2\gen{2}{0}{2}{4}}{g\ggen{0}{6}{15}{6}{2}}{12},\nn\\
\gen{1}{9}{29}{0} &= \tv{f}{\gen{0}{9}{25}{4}}{4},&
\gen{0}{10}{29}{2} &= \tv{g}{(\gen{0}{3}{8}{2})^3}{5},\nn\\
\gen{0}{10}{30}{0} &= \tv{g}{(\gen{0}{3}{8}{2})^3}{6},&
\gen{2}{9}{31}{0}
 &= \tv{f^2}{\gen{0}{3}{8}{2}\ggen{0}{6}{15}{6}{1}}{8},\nn\\
\gen{1}{10}{31}{2} &= \tv{f}{(\gen{0}{5}{14}{2})^2}{3},&
\gen{1}{10}{32}{0} &= \tv{f}{(\gen{0}{5}{14}{2})^2}{4},\nn\\
\gen{3}{9}{33}{0}
 &= \tv{f\gen{2}{0}{2}{4}}{\gen{0}{3}{8}{2}\ggen{0}{6}{15}{6}{2}}{8},&
\gen{1}{11}{35}{0} &= \tv{f}{\gen{0}{3}{8}{2}\gen{0}{8}{23}{2}}{4},\nn\\
\gen{0}{12}{35}{2} &= \tv{\gen{0}{3}{5}{8}}{(\gen{0}{3}{8}{2})^3}{6},&
\gen{2}{11}{37}{0}
 &= \tv{f^2}{\gen{0}{3}{6}{6}\gen{0}{8}{23}{2}}{8},\nn\\
\gen{1}{13}{41}{0} &= \tv{f}{\gen{0}{3}{8}{2}\gen{0}{10}{29}{2}}{4},&
\gen{0}{15}{45}{0} &= \tv{\gen{0}{3}{5}{8}}{(\gen{0}{3}{8}{2})^4}{8}.
\end{align}
%

%%%%%%%%%%%%%%%%%%%%%%%%%%%%%%%%%%%%%%%%%%%%%%%%%%%%%%%%%%%%%%%%%%%%%%%%
\section{Special functions}\label{app:functions}
%%%%%%%%%%%%%%%%%%%%%%%%%%%%%%%%%%%%%%%%%%%%%%%%%%%%%%%%%%%%%%%%%%%%%%%%

The Jacobi theta functions are defined as
\begin{align}
\begin{aligned}
\varth_1(z,\tau)&:=
 \ri\sum_{n\in\bbZ}(-1)^n y^{n-1/2}q^{(n-1/2)^2/2},\\
\varth_2(z,\tau)&:=
  \sum_{n\in\bbZ}y^{n-1/2}q^{(n-1/2)^2/2},\\
\varth_3(z,\tau)&:=
  \sum_{n\in\bbZ}y^n q^{n^2/2},\\
\varth_4(z,\tau)&:=
  \sum_{n\in\bbZ}(-1)^n y^n q^{n^2/2},
\end{aligned}
\end{align}
where
\begin{align}
y=e^{2\pi \ri z},\qquad q=e^{2\pi \ri\tau}
\end{align}
and $z\in\bbC, \tau\in\bbH$.
We often use the abbreviated notation
\begin{align}
\varth_k(\tau):=\varth_k(0,\tau).
\end{align}
The Dedekind eta function is defined as
\begin{align}
\eta(\tau):=q^{1/24}\prod_{n=1}^\infty(1-q^n).
\end{align}
The Eisenstein series are given by
\begin{align}
E_{2n}(\tau)
 =1-\frac{4n}{B_{2n}}\sum_{k=1}^{\infty}\frac{k^{2n-1}q^k}{1-q^k}
\label{eq:eisendef}
\end{align}
for $n\in\bbZ_{>0}$. The Bernoulli numbers $B_k$ are defined by
\begin{align}
\frac{x}{e^x-1}=\sum_{k=0}^\infty\frac{B_k}{k!}x^k.
\end{align}
We often abbreviate $\eta(\tau),\,E_{2n}(\tau)$
as $\eta,\,E_{2n}$ respectively.

%%%%%%%%%%%%%%%%%%%%%%%%%%%%%%%%%%%%%%%%%%%%%%%%%%%%%%%%%%%%%%%%%%%%%%%%
%%%%%%%%%%%%%%%%%%%%%%%%%%%%%%%%%%%%%%%%%%%%%%%%%%%%%%%%%%%%%%%%%%%%%%%%
%%%%%%%%%%%%%%%%%%%%%%%%%%%%%%%%%%%%%%%%%%%%%%%%%%%%%%%%%%%%%%%%%%%%%%%%
%%% References %%%

\renewcommand{\section}{\subsection}
\renewcommand{\refname}

\end{document}